\documentclass[11pt]{amsart}
\usepackage{fullpage,amssymb,amsmath,verbatim,amstext,eucal,amscd}

\usepackage{graphicx}

\usepackage{epsfig}
\RequirePackage{color}
\usepackage[normalem]{ulem}
\usepackage[all]{xy}
\usepackage{enumerate}

\usepackage{color}
\usepackage{fullpage}

\numberwithin{equation}{section}

\newtheorem{thm}[equation]{Theorem}

\newtheorem{lemma}[equation]{Lemma}
\newtheorem{prop}[equation]{Proposition}
\newtheorem{cor}[equation]{Corollary}

\newtheorem{proposition}[equation]{Proposition}
\newcommand{\FLEX}{\relax}
\newcommand{\flex}[1]{\renewcommand{\FLEX}{#1}}
\newtheorem{flexthm}[equation]{\FLEX}

\newenvironment{dremark}[1]{\refstepcounter{equation}
\vskip 5pt \par\noindent {\bf #1\ \thethm .}}{\vskip 5pt \par}
\newenvironment{dremark*}[1]{\vskip 5pt \par\noindent {\bf #1.}}{\vskip 5pt \par}

\theoremstyle{definition}
\newtheorem{definition}[equation]{Definition}

\theoremstyle{remark}
\newtheorem{remark}[equation]{Remark}

\newtheorem{example}[equation]{Example}

\newcommand{\dstext}[1]{\quad\text{#1}\quad}
\newcommand{\innerprod}[1]{\left\langle #1\right\rangle}

\newcommand{\aut}{\operatorname{Aut}}

\newcommand{\coact}{\nu}

\def\go{G^{(0)}}
\newcommand{\unit}[1]{#1^{(0)}}
\newcommand{\norm}[1]{\left\|{#1}\right\|}

\def\cs{C^{\ast}}
\DeclareMathOperator{\supp}{supp}
\DeclareMathOperator{\spn}{span}
\DeclareMathOperator{\ran}{ran}
\DeclareMathOperator{\dom}{dom}

\DeclareMathOperator{\id}{id}
\DeclareMathOperator{\gr}{gr}

\newcommand{\N}{\mathbb{N}}
\newcommand{\Z}{\mathbb{Z}}
\newcommand{\inv}{^{-1}}
\newcommand{\C}{\mathbb{C}}
\newcommand{\T}{\mathbb{T}}
\newcommand{\Rho}{\mathcal{P}}
\newcommand{\ol}{\overline}
\newcommand{\sE}{\mathcal{E}}

\newcommand{\sG}{\mathcal{G}}
\newcommand{\sH}{\mathcal{H}}

\newcommand{\fS}{\mathfrak{S}}
\newcommand{\fs}{\mathfrak{s}}
\newcommand{\fr}{\mathfrak{r}}

\newcommand{\bbC}{{\mathbb{C}}}

\newcommand{\bbN}{{\mathbb{N}}}

\newcommand{\bbT}{{\mathbb{T}}}

\newcommand{\eps}{\varepsilon}

\newcommand{\<}{\langle}
\renewcommand{\>}{\rangle}
\newcommand{\pair}{(\Sigma; G)}
\newcommand{\cpair}{C^*_r(\Sigma; G)}
\newcommand{\ccpair}{C_c(\Sigma; G)}
\newcommand{\gc}{$\Gamma$-Cartan} 
\newcommand{\cspn}{\ol{\spn}}

\newcommand{\dual}[1]{#1^\#}
\newcommand{\cstar}{\hbox{$C^*$}}
\newcommand{\cstaralg}{$C^*$-algebra}

\title[Graded $\cs$-algebras and twists]{Graded $\cs$-algebras and twisted groupoid $\cs$-algebras}

\author[J.H. Brown]{Jonathan H. Brown}
\address[J.H. Brown]{
Department of Mathematics\\
University of Dayton\\
300 College Park Dayton\\
OH 45469-2316 U.S.A.} \email{jonathan.henry.brown@gmail.com}

\author[A.H. Fuller]{Adam H. Fuller}
\address[A.H. Fuller]{
Department of Mathematics\\
Ohio University\\
Athens\\
OH 45701 U.S.A.}
\email{fullera@ohio.edu}

\author[D.R. Pitts]{David R. Pitts}
 \address[D.R. Pitts]{
  Department of Mathematics\\
  University of Nebraska-Lincoln\\
  Lincoln\\
  NE 68588-0130 U.S.A.}  \email{dpitts2@math.unl.edu}

\author[S.A.  Reznikoff]{Sarah A. Reznikoff}
\address[S. A. Reznikoff]{
Department of Mathematics\\
Kansas State University\\
138 Cardwell Hall\\
Manhattan, KS, U.S.A. }
\email{sarahrez@math.ksu.edu}

\begin{document}

\begin{abstract}	
  Let $\cs$-algebra that is acted upon by a compact abelian
  group.  We show that if the fixed-point algebra of the action
  contains a Cartan subalgebra $D$ satisfying an appropriate
  regularity condition, then $A$ is the reduced $\cs$-algebra of a
  groupoid twist.  We further show that the embedding
  $D \hookrightarrow A$ is uniquely determined by the twist.  These
  results generalize Renault's results on Cartan subalgebras of
  $\cs$-algebras.
\end{abstract}
\maketitle

\section{Introduction}

Abelian operator algebras are well understood: abelian $C^*$-algebras
are all isomorphic to spaces of continuous functions on a locally compact
Hausdorff space; abelian von Neumann algebras are all isomorphic to
$L^\infty$-spaces.  The study of non-abelian operator algebras is
often aided by the presence of appropriate abelian subalgebras.  This idea was
exemplified by Feldman and Moore's characterization of von Neumann
algebras containing Cartan subalgebras in 1977 \cite{FM77}.  Cartan
embeddings arise naturally in many examples, including finite
dimensional von Neumann algebras and von Neumann algebras constructed
from free actions of discrete groups on abelian von Neumann algebras.
Feldman and Moore \cite{FM77} gave a complete classification of Cartan
subalgebras in terms of measured equivalence relations.

To transfer Feldman and Moore's theory to the topological setting,
Renault \cite{Ren08} defined Cartan subalgebras for $C^*$-algebras.

\begin{definition}\cite[Definition~5.1]{Ren08} Let $A$ be a
$C^*$-algebra.  A maximal abelian $C^*$-algebra $D \subseteq A$ is a
\emph{Cartan subalgebra} of $A$ if
\begin{enumerate}
\item there exists a faithful conditional expectation $E \colon A
\rightarrow D$;
\item $D$ contains an approximate unit for $A$;
\item the set of normalizers of $D$, i.e.~the $n\in A$ such that
$nDn^*\in D$ and $n^*Dn \in D$, generate $A$ as a $C^*$-algebra.
\end{enumerate} When $D$ is a Cartan subalgebra of $A$, we call
$(A,D)$ a \emph{Cartan pair}.
\end{definition} Renault \cite{Ren08}, building on work by Kumjian
\cite{Kum86}, showed that there is a one-to-one correspondence between
Cartan pairs of separable $C^*$-algebras and $C^*$-algebras of second countable twisted groupoids; that is, between
Cartan pairs and the reduced $C^*$-algebra generated by an extension
of groupoids \[\T \times \go  \rightarrow \Sigma \rightarrow
G.\] In Renault's result, $G$ must be topologically principal: Renault refers to $G$ as the Weyl groupoid of the Cartan pair.  It is
reasonable to seek a larger class of inclusions $D\subseteq
A$ with $D$ abelian that can be used to construct twists.

This idea has recently successfully been pursued by several authors,
and larger classes of inclusions have been shown to arise as \cstaralg
s of twists.  In particular, motivated by shift spaces and the work by
Matsumoto and Matui \cite{MM14, Mat12, Mat13}, Brownlowe, Carlsen, and
Whittaker \cite{BCW17} were able to construct a Weyl type groupoid
from a general graph $C^*$-algebra and its canonical diagonal and use
this construction to show that diagonal-preserving isomorphisms of
these inclusions come precisely from isomorphisms of Weyl type
groupoids.  This led to work proving similar results for Leavitt path
algebras \cite{BCH} and Steinberg algebras \cite{ABHS16}.

The paper \cite{ABHS16} in particular inspired this work.  Steinberg
algebras are algebraic analogues of groupoid $C^*$-algebras
\cite{Steinberg, CFST}.  In \cite{ABHS16}, the authors consider
Steinberg algebras associated to groupoids $G$ equipped with a
homomorphism $c: G\to \Gamma$ where $\Gamma$ is an abelian group and
$c\inv(0)$ is topologically principal.  The Steinberg algebra is then
naturally graded by $\Gamma$; the authors use this grading to
reconstruct $G$.  It is well known that algebras graded by an abelian
group $\Gamma$ correspond in the $C^*$-algebraic theory to
$C^*$-algebras endowed with a $\hat\Gamma$ action (for example see
\cite{Tom07}, \cite{Rae18}).

In this paper, we construct groupoids
from inclusions of an abelian $C^*$-algebra $D$ into a $C^*$-algebra
$A$ endowed with the action of a compact abelian group.  
In particular, the aim of our work is to generalize Renault's
characterization of Cartan pairs by reduced $C^*$-algebras of twisted
groupoids.  Our results apply to examples appearing naturally in the
study of higher-rank graph and twisted higher-rank graph
$C^*$-algebras (See Example~\ref{THRG} below).

We start with a $C^*$-algebra $A$ and a discrete abelian countable
group $\Gamma$ such that the dual group $\hat{\Gamma}$ acts
continuously on $A$ by automorphisms.  
Let $A^{\hat{\Gamma}}$ be the points in $A$ fixed by the action of
$\hat{\Gamma}$: this is a subalgebra of $A$  called the fixed point
algebra.  Assume $A^{\hat{\Gamma}}$ contains a Cartan subalgebra $D$.
If in addition the normalizers of $D$
in $A$ densely span $A$ we call $(A,D)$ a \emph{\gc\ pair}.
We note that the normalizers of $D$ in $A^{\hat{\Gamma}}$ are homogeneous of
degree $0$.  In particular, if the action by $\hat{\Gamma}$ is
trivial, then $(A,D)$ is a Cartan pair.

If $(A,D)$ is a \gc\ pair, then following Kumjian's construction, we
show how to 
create a twisted groupoid $\pair$ that is graded by $\Gamma$.
This yields the following commutative diagram
\begin{equation}\label{introdiag} \xymatrix{ \T\times \go\ar[r]^{\iota} &\Sigma
\ar[dr]_{c_{_\Sigma}} \ar[r]^q & G \ar[d]^{c_{_G}} \\
& & \Gamma& }
\end{equation}
where $c_\Sigma$ and $c_G$ are homomorphisms.  We
prove in Theorem~\ref{mainthm1} that there is a natural isomorphism
between $(A,D)$ and the reduced crossed product
$(C_r^*(\Sigma;G),C_0(G^{(0)}))$.

Next, if $\Sigma\to G$ is a twist satisfying the  commutative
diagram~\eqref{introdiag}, we show
that the inclusion
$C_0(\go)\hookrightarrow C^*_r(\Sigma;G)$ satisfies our hypotheses, so
Theorem~\ref{mainthm1} allows us to construct a new twist from this
inclusion.  The natural question is: does our construction in
Theorem~\ref{mainthm1} recover $\Sigma\to G$?  We answer this
affirmatively in Theorem~\ref{mainthm2}.  This second question is the
main focus of \cite{BCW17}, \cite{BCH} \cite{ABHS16}, and \cite{CRST}
in the case that the twist is trivial.

The paper~\cite{CRST} by Carlsen, Ruiz, Sims and Tomforde is similar
in scope to our present work.  While~\cite{CRST} is also concerned
with translating the results of \cite{ABHS16} to a $C^*$-algebraic
framework, their work avoids twists altogether, instead focusing on
showing rigidity results along the lines of our
Theorem~\ref{mainthm2}.  The results of~\cite{CRST} apply to
$C^*$-algebras already known to arise from groupoids, however it does
contain some remarkable innovations which allows the authors to address
$C^*$-algebras endowed with co-actions of a possibly nonabelian group. Furthermore,
Carlsen, Ruiz, Sims and Tomforde relax the requirement that the
abelian subalgebra $D$ must be Cartan in the fixed point algebra.

Whether or not a $\cs$-algebra satisfies the Universal Coefficient Theorem (UCT) remains the main stumbling block in the classification program for simple nuclear $\cs$-algebras. 
Indeed, Tikuisis, White and Winter \cite{TWW} have shown that all separable, unital, simple, nuclear $\cs$-algebras with finite nuclear dimension that satisfying the UCT are classifiable.  Recent results of Barlak and Li \cite{BarlakLi} show that if $A$ is a nuclear $\cs$-algebra containing a Cartan subalgebra, then $A$ satisfies the UCT.
We discuss in Example~\ref{ex: BL} how their results also apply to our setting.


This paper is organized as follows.  We begin with preliminaries
on twists (Section~\ref{prelim}). In Section~\ref{gamma Cartan} we define $\Gamma$-Cartan pairs and review the relationship between topological grading and strong group actions.

In Section~\ref{RI} we prove our
main theorem, which shows that a large class of $C^*$-algebras are
isomorphic to the reduced $C^*$-algebra of a twist. In Section~\ref{CoT} we then provide a
few basic results concerning a natural \gc\ pair that arises in the
presence of a twist.  In Section~\ref{RD} we
prove our rigidity result, Theorem~\ref{mainthm2}, which shows that if
the inclusion in the previous section comes from a twist then our
construction recovers the twist.

Section~\ref{ex} gives some examples to which our theorems
apply. Notably, in Example~\ref{THRG} we show that the twisted
higher-rank graph $\cs$-algebras introduced in \cite{KPS12} and
\cite{KPS15} give examples of \gc\ pairs.  Moreover, the groupoid
description of twisted higher-rank graph $\cs$-algebras given in
\cite{KPS15} yields groupoids isomorphic to ours.

Finally, in an appendix, we describe how we can obtain the results of Section~\ref{RI} by using a coaction of a non-abelian group (instead of an action of an abelian group); note that in this case the grading on the $C^*$-algebra is by the group itself, rather than by its dual.
The authors thank John Quigg for pointing out this alternative construction.

\subsection*{Acknowledgments} Whilst conducting this research JHB and
AHF made use of meeting space made available by the Columbus
Metropolitan Library.  We would like to thank CML for their important
work in the community.  AHF would like to thank Christopher Schafhauser for patiently
answering his questions on \cite{BarlakLi}.

This work was partially supported by grants from the Simons Foundation (\#316952 to David Pitts and  \#360563 to Sarah Reznikoff) and by the American Institute of Mathematics SQuaREs Program.

\section{Preliminaries}\label{prelim}

\subsection{\'Etale groupoids} A groupoid $G$ is a small category in which every morphism has an inverse.  The unit space $\go$ of $G$ is
the set of identity morphisms.  The maps $s,r:G\to \go$, given by $s(\gamma)=
\gamma^{-1}\gamma$ and $r(\gamma)= \gamma\gamma^{-1}$, are the
\textit{source} and \textit{range} maps.   For $S,T\subseteq G$ we denote
\[ST:=\{\gamma\eta: \gamma\in S, \eta\in T, r(\eta)=s(\gamma)\}.\] If
either $S$ or $T$ is the singleton set $\{\gamma\}$ we remove the set
brackets from the notation and write $S\gamma$ or $\gamma T$.

A topological groupoid is a groupoid $G$ endowed with a topology such
that inversion and composition are continuous.  An open set $B\subseteq
G$ is a {\it bisection} if $r(B)$ and $s(B)$ are open and $r|_B$ and
$s|_B$ are homeomorphisms onto their images.  The groupoid $G$ is
\textit{\'etale} if there is a basis for the topology on $G$
consisting of bisections.  When $G$ is \'etale then $\go$ is
open and closed in $G$.

For $x\in \unit{G}$, the \textit{isotropy group at $x$} is
$xGx:=\{\gamma\in G: r(\gamma)=s(\gamma)=x\}$ and the
\textit{isotropy subgroupoid} is the set $G'=\{\gamma\in G:
r(\gamma)=s(\gamma)\}$.  A topological groupoid $G$ is {\it
topologically principal} if $\{x\in \unit{G}: xGx=\{x\}\}$ is dense
in $\unit{G}$; it is {\em effective} if the interior of $G'$ is $\go$.   If $G$ is second countable these notions coincide \cite[Lemma~3.1]{BCFS}, but in the general (not necessarily second countable) case, effective is the more useful notion.

 Unless explicitly stated otherwise, for the remainder of this paper,  we make the following assumptions.
\begin{dremark*}{Standing Assumptions on Groupoids}
  \textit{Throughout, all groupoids are:\begin{enumerate}
    \item locally compact and \item
      Hausdorff. 
\end{enumerate}
}

\end{dremark*}
\subsection{Twists} The main focus of this paper is on twists and
their $\cs$-algebras.  We provide a brief account of the necessary
background here.  Much of this background can also be found in
\cite{Ren08}.   We also encourage the reader to consult the recent
expository article by Sims~\cite{SimsHaEtGrThC*Al}.
We now expand on a few details that are
particularly relevant to our context.

A twist is the analog of a central extension of a discrete group by
the circle $\bbT$.  Here is the formal definition.

\begin{definition}[see~{\cite[Definition~5.1.1]{SimsHaEtGrThC*Al}}]\label{twistdef} Let
  $\Sigma$ and $G$ be topological groupoids with $G$ \'etale, and let
  $\bbT\times \unit{G}$ be the product groupoid. That is,
  $(z_1,x_1)(z_2,x_2)$ is defined if and only if $x_1=x_2$, in which
  case the product is given by  $(z_1,x_1)(z_2,x_2)=(z_1z_2,x_1)$;
  inversion is $(z,x)^{-1}=(z^{-1},x)$, and the
  topology is the product topology.  The unit space of
  $\bbT\times\unit{G}$ is $\{1\}\times \unit{G}$.
  
  The pair
  $(\Sigma,G)$ is a \textit{twist} if there is an exact sequence
  \[\bbT\times\unit{G}\overset{\iota}\rightarrow
    \Sigma\overset{q}\rightarrow G\] where
  \begin{enumerate}
    \item $\iota$ and $q$ are continuous groupoid homomorphisms with
      $\iota$ one-to-one and $q$ onto;
      \item $\iota|_{\{1\}\times \unit{G}}$ and $q|_{\unit{\Sigma}}$
        are homeomorphisms onto $\unit{\Sigma}$ and $\unit{G}$,
        respectively (identify $\unit{\Sigma}$ and $\unit{G}$ using $q$);
      \item  $q^{-1}(\unit{G})=\iota(\bbT\times
        \unit{G})$;
       \item  for every $\gamma\in \Sigma$ and $z\in\bbT$, 
         $\iota(z,r(\gamma))\gamma =\gamma \iota(z,s(\gamma))$; and
         \item   for every $g\in G$ there is an open bisection $U$
           with $g\in U$ and a
           continuous 
           function $\phi_U: U\rightarrow \Sigma$ such that $q\circ
           \phi_U=\id|_U$ and the map $\bbT\times U\ni (z\times h)\mapsto
           \iota(z,r(h))\, \phi_U(h)$ is a homeomorphism of $\bbT\times U$
           onto $q^{-1}(U)$.
       \end{enumerate}
       (Conditions (1--3) say the sequence is an extension, (4) says
       the extension is central, and (5) says $G$ is \'etale and
       the extension is locally trivial.)
       A twist is often denoted simply by $\Sigma \rightarrow G$.

       For $z\in\bbT$ and $\gamma\in\Sigma$ we will write
       \[z\cdot \gamma:=\iota(z,r(\gamma))\gamma\dstext{and} \gamma
         \cdot z:=\gamma\iota(z,s(\gamma))\] for the action of $\bbT$
       on $\Sigma$ arising from the embedding of $\bbT\times \unit{G}$
       into $\Sigma$. Notice that this action of $\bbT$ on $\Sigma$ is free.

Also, for $\gamma\in \Sigma$, we will often denote $q(\gamma)$
by $\dot \gamma$; indeed, we use the name $\dot\gamma$ for an
arbitrary element of $G$.
       
     \end{definition}

\begin{remark}\label{qmap}
  By~\cite[Exercise~9K(3)]{WillardGeTo} the map
  $q:\Sigma\rightarrow G$ is a quotient map.
\end{remark}


The $C^*$-algebra of the twist is constructed from the completion of
an appropriate function algebra $C_c(\Sigma;G)$.  This algebra can be
constructed in two different ways and both will be used in this note.

\subsubsection*{First description of $C_c(\Sigma;G)$: Sections of a
  line bundle}
The first way to construct $C_c(\Sigma;G)$ is by considering sections
of a complex line bundle $L$ over $G$.  Define $L$ to be the quotient
of $\bbC\times \Sigma$ by the 
equivalence relation on $\bbC\times \Sigma$ given by 
$(\lambda,\gamma)\sim (\lambda_1,\gamma_1)$ if and only if there
exists $z\in\bbT$ such that
$(\lambda_1,\gamma_1)=(\overline{z}\lambda, z\cdot \gamma)$.   We
sometimes write $L=(\bbC\times \Sigma)/\bbT$.  Use
$[\lambda, \gamma]$ to denote the equivalence class of $(\lambda,
\gamma)$.  Observe that for any $z\in\bbT$,
\begin{equation}\label{Tact}
  [\lambda, z\cdot \gamma]=[z\lambda, \gamma].
\end{equation}
With the quotient
topology, $L$ is Hausdorff.  The (continuous) surjection $P:L\rightarrow G$ is given
by
\[ P:[\lambda, \gamma] \mapsto \dot \gamma.
\]
For  $\dot\gamma\in G$ and
$\gamma_0\in q^{-1}(\dot\gamma)$, the map $\bbC\ni\lambda\mapsto
[\lambda, \gamma_0]\in P^{-1}(\dot\gamma)$ is a homeomorphism, so $L$ is a complex line bundle over $G$.  In
general, there is no canonical choice of $\gamma_0$.  However, when
$\dot\gamma\in \unit{G}$, $\unit{\Sigma}\cap q^{-1}(\dot\gamma)$ is a
singleton set,  so there is a canonical choice:
take $\gamma_0$ to be the unique element of 
$\unit{\Sigma}\cap q^{-1}(\dot\gamma)$.  Thus, recalling that
$\unit{\Sigma}$ and $\unit{G}$ have been previously identified (using
$q|_{\unit{\Sigma}}$), when $x\in\unit{G}=\unit{\Sigma}$, we
sometimes identify $P^{-1}(x)$ with $\bbC$ via the map $\lambda\mapsto
[\lambda, x]=\lambda\cdot [1,x]$.

Finally, there is
a  continuous map $\varpi: L\rightarrow [0,\infty)$ given by \[\varpi([\lambda,
\gamma]):=|\lambda|.\] 
When $f: G\rightarrow L$ is a section and $\dot\gamma\in G$, we will sometimes write
$|f(\dot\gamma)|$ instead of $\varpi(f(\dot\gamma))$.

Since $\Sigma$ is locally trivial, $L$ is locally trivial as well.
Indeed, given $\ell\in L$, let $B$ be an open bisection of $G$
containing $P(\ell)$.  Let $\phi_B:B\rightarrow \Sigma$ be
a continuous function satisfying the conditions of 
Definition~\ref{twistdef}(5).  Then for every element  $\ell_1\in
P^{-1}(B)$, there exist unique  $\lambda\in\bbC$ and $\dot\gamma\in B$
so that  $\ell_1=[\lambda, \phi_B(\dot\gamma)]$.  It follows that  the map
$[\lambda,\phi_B(\dot\gamma)]\mapsto (\lambda, \dot\gamma)$ is a
homeomorphism of $P^{-1}(B)$ onto $\bbC\times B$, so $L$ is locally trivial.

There is a partially defined  multiplication on $L$, given by
\[
[\lambda, \gamma][\lambda',
\gamma']=[\lambda\lambda',\gamma\gamma'],\] whenever $\gamma$ and
$\gamma'$ are composable in $\Sigma$.  When $[\lambda,\gamma],
[\lambda',\gamma']\in L$ satisfy $\dot\gamma=\dot\gamma'$, let
\begin{equation}\label{adds}
  [\lambda,\gamma]+[\lambda',\gamma']:=[\lambda +z\lambda',\gamma],
\end{equation}
where $z$ is the unique element of $\bbT$ so that $\gamma'=z\cdot\gamma$.  There is also an involution on
$L$ given by
\begin{equation}\label{invsg}
  \overline{[\lambda,\gamma]}=[\overline{\lambda},\gamma\inv] .
\end{equation}

We use the symbol $C_c(\Sigma;G)$ to denote the set of ``compactly
supported'' continuous sections of $L$, that is, 
\begin{equation}\label{ccsg1} C_c(\Sigma;G):=\{f:G\rightarrow L \, |\,  f \text{ is continuous, $P\circ
  f=\id|_G$, and 
  $\varpi\circ f$ has compact support}\}.
\end{equation}

\begin{dremark}{Notation} For $f\in C_c(\Sigma; G)$,
we denote the support of $\varpi\circ f$ by $\supp(f)$; we denote
its open support by $\supp'(f)$. Further, let $C(\Sigma;G)$ and
$C_0(\Sigma;G)$ be, respectively, the set of continuous sections and
continuous sections vanishing at infinity of the bundle $L$.  
\end{dremark}

We endow $C_c(\Sigma;G)$ with a $*$-algebra structure where addition is
pointwise (using~\eqref{adds}), multiplication is given by convolution:
\begin{equation}\label{secmul}
  f*g(\dot \gamma)=\sum_{\dot\eta_1\dot\eta_2=\dot\gamma}f(\dot\eta_1)g(\dot\eta_2)=\sum_{r(\dot\eta)=r(\dot\gamma)}
f(\dot\eta)g(\dot\eta\inv \dot \gamma),
\end{equation} and the involution is from~\eqref{invsg}:
\[f^*(\dot \gamma)=\overline{f(\dot\gamma^{-1})}.\]
Note that if $f,g$ are supported on bisections $B_1,B_2$ and $\dot \eta_i\in B_i$ then $f*g(\dot\eta_1\dot\eta_2)=f(\dot\eta_1)g(\dot\eta_2)$.
We can identify $C_0(\go)$ with a subalgebra of continuous sections of
the line bundle $L$ by
\[ C_0(\go)\to C_0(\Sigma;G)\quad\text{by}\quad \phi\mapsto
\left(\dot\gamma\mapsto \begin{cases}
[\phi(\dot\gamma),\iota(1,\dot\gamma)] & \dot\gamma\in \go\\ 0
&\text{otherwise}\end{cases}\right).\] Note that this identification takes
pointwise multiplication on $C_0(\go)$ to the convolution on
$C_c(\Sigma;G)$.

\subsubsection*{Second description of $C_c(\Sigma;G)$: Covariant
  functions}  A function $f$ on $\Sigma$ is  \textit{covariant} if for every $z\in
\bbT$ and $\gamma\in \Sigma$,
\[f(z\cdot\gamma)=\overline{z}\, f(\gamma).\]
The second way to describe $C_c(\Sigma;G)$ is as the set of
compactly supported continuous covariant functions on $\Sigma$, that
is, 
\begin{equation}\label{ccsg2}
 C_c(\Sigma;G):=\{f\in C_c(\Sigma): \forall \gamma \in \Sigma \, \,
\, \forall z\in \T\ \, f(z\cdot\gamma)=\overline{z}f(\gamma)\}.
\end{equation}
Addition is pointwise, the involution is
$f^*(\gamma)=\overline{f(\gamma^{-1})}$, and the  convolution
multiplication is given by
\begin{equation}\label{covmul}
  f*g(\gamma)=\sum_{\substack{\dot\eta \in G \\
r(\dot\eta)=r(\dot\gamma)}} f(\eta)g(\eta\inv \gamma),
\end{equation} where for each $\dot\eta$ with
$r(\dot\eta)=r(\dot\gamma)$, only one representative $\eta$ of
$\dot\eta$ is chosen.  It is easy to verify that this is
well-defined. 

\subsubsection*{Equivalence of the descriptions}
To proceed, we need to be more explicit on how these two descriptions
of $C_c(\Sigma;G)$ are the same.  Take $f\in C_c(\Sigma)$ such that
$f(z\cdot \gamma)=\overline{z} f(\gamma)$ for all $\gamma \in \Sigma$ and
$z \in \T$.  Let $\tilde{f}$ be the section of the line bundle given
by
\[ \tilde{f}(\dot\gamma)=[f(\gamma),\gamma].
\] Note that by the definition of the line bundle, this is
well-defined.

On the other hand, consider a compactly supported continuous section
$\tilde{f}:G\rightarrow L$.  For $\gamma\in\Sigma$, the fact that
$P\circ\tilde f=\id|_G$ yields
$P\left([1,\gamma]^{-1}\tilde f(\dot\gamma)\right)=s(\dot\gamma)$.
Hence there exists $\lambda_\gamma\in\bbC$ such that
$[1,\gamma]^{-1}\tilde f(\dot\gamma)=\lambda_\gamma\cdot
[1,s(\gamma)]$, that is,
\[\tilde f(\dot\gamma)= \lambda_\gamma\cdot [1,
  \gamma]=[\lambda_\gamma,\gamma].\] 
Define $f:\Sigma\rightarrow \bbC$ by 
\[f(\gamma)=\lambda_\gamma.\] Then $f$ is continuous and
compactly supported since $\tilde{f}$ is and  satisfies
\begin{equation}\label{cov}
  f(z\cdot \gamma)=\overline{z}f(\gamma).
\end{equation}
We have thus described a linear
isomorphism between the spaces defining 
   $C_c(\Sigma;G)$ given in \eqref{ccsg1} and \eqref{ccsg2}.   It is
  a routine matter to show this linear map is a $*$-algebra
  isomorphism, so that the two descriptions coincide.  Notice that $\gamma\in \supp(f)$ if and only if $\dot\gamma\in \supp (\tilde{f})$.

\begin{remark}\label{supp covariant} Technically, the support of a function $f:\Sigma\rightarrow \bbC$
 satisfying the covariance condition \eqref{cov} is a subset of $\Sigma$, but~\eqref{cov}
 allows us to regard both $\supp(f)$ and $\supp'(f)$ as subsets of
 $G$.  We shall do this.  Thus the notions of support are the same
 whether $f$ is viewed as a covariant function or as
 a section of the line bundle.
\end{remark}

To define the reduced groupoid $C^*$-algebra, we need to define
regular representations.  For $x \in G^{(0)}$, let $\mathcal{H}_x=\ell^2(\Sigma
x,Gx)$ be the set of square summable sections of the line bundle
$L|_{Gx}$; that is, \[\mathcal{H}_x=\{\chi: Gx\rightarrow  P^{-1}(Gx) \, |\,
 \text{ for $\dot\gamma\in Gx$, } P(\chi(\dot\gamma))=\dot\gamma, \text{
   and } \varpi\circ \chi\in \ell^2(Gx)\}.\]
Given $\chi_1, \chi_2\in \mathcal{H}_x$ and $\dot\gamma
\in Gx$,  $P\left(\overline{\chi_2(\dot\gamma)}\chi_1(\dot\gamma)\right)=x\in
\unit{G}$, so that we obtain a unique
$\lambda_{\dot\gamma}\in\bbC$ so that
\[\overline{\chi_2(\dot\gamma)}\chi_1(\dot\gamma))=\lambda_{\dot\gamma}\cdot
  [1,x].\]  We may therefore define an inner product on $\sH_x$:
$\innerprod{\chi_1,\chi_2}$ is the unique element of $\bbC$ such that
\begin{equation}\label{ipdef}
  \sum_{\dot\gamma\in
    Gx}\overline{\chi_2(\dot\gamma)}\chi_1(\dot\gamma)=\innerprod{\chi_1,\chi_2}\cdot
  [1,x]=[\innerprod{\chi_1,\chi_2},x].
\end{equation}
  
The regular representation of $\ccpair$ on $\mathcal{H}_x$ is then
defined as follows.   For $f\in C_c(\Sigma,G)$ and $\chi\in \sH_x$,
\[
  \left(\pi_x(f)\chi\right)(\dot\gamma)= \sum_{\stackrel{\dot\eta\in
      G, \dot\zeta \in Gx}{\scriptscriptstyle
      \dot\eta\dot\zeta=\dot\gamma}} f(\dot\eta)\chi(\dot\zeta)=
  \sum_{\stackrel{\dot\eta\in G}{\scriptscriptstyle r(\dot\eta)=r(\dot\gamma)}}
  f(\dot\eta)\chi(\dot\eta\inv\dot\gamma) \qquad (\dot\gamma\in Gx).
\]

The reduced \cstaralg\ of $(\Sigma;G)$, denoted $C_r^*(\Sigma;G)$, is the
completion of $\ccpair$ under the norm
$\|f\|_r=\sup_{x\in \go} \|\pi_x(f)\|.$

\begin{remark}\label{altpix}   
 Viewing  $C_c(\Sigma,G)$ as  the space of compactly supported sections of the line bundle affords us an alternative way to describe the regular representations, as follows. Given $x \in G^{(0)}$, define a linear functional $\epsilon_x$   on  $C_c(\Sigma,G)$   by defining $\epsilon_x(f)$ to be the unique scalar  such that $f(x)=[\epsilon_x(f), x]$. Note that for $f \in C_c(\Sigma,G)$, \[\epsilon_x(f^*f)=\sum_{s(\dot{\gamma})=x} \varpi(f(\dot \gamma))^2 \geq 0, \] so $\epsilon_x$ is positive.  Morever, if also $g \in C_c(\Sigma, G)$, then \cite[Proposition~3.10]{Exel08} shows
  there exist a finite number of open bisections $U_1, \dots, U_n$ for
  $G$ such that $\supp(g)\subseteq \bigcup_{j=1}^n U_j$, from which it
  follows that
  $(\eps_x(f^*g^*g f))^{1/2}\leq n\norm{g}_\infty \eps_x(f^*f)^{1/2}$.
  Thus the GNS construction may be applied to $\eps_x$ to produce a
  representation $(\pi_{\eps_x}, \sH_{\eps_x})$ of $C_c(\Sigma;G)$. Letting $L_{\eps_x}$ be the left kernel of $\eps_x$, the map
  $C_c(\Sigma;G)/L_{\eps_x} \ni g+L_{\eps_x}\mapsto g|_{Gx}$ is
  isometric and so determines an isometry
  $W:\sH_{\eps_x}\rightarrow \sH_x$.  As $G$ is \'etale, for
  $\dot\gamma\in Gx$, there is an open bisection $U$ for $G$ with
  $\dot\gamma\in U$, and hence we may find $g\in C_c(\Sigma;G)$
  supported in $U$ with $g(\dot\gamma)\neq 0$.   Thus, if
  $h\in \ell^2(\Sigma x,Gx)$ has finite support, there exists
  $f\in C_c(\Sigma,G)$ with $f|_{Gx}=h$.  This implies that $W$ is
  onto, and a calculation shows that $W\pi_{\eps_x}=\pi_x W$.  This
  shows $\pi_x$ and $\pi_{\eps_x}$ are unitarily equivalent
  representations of $C_c(\Sigma;G)$.  Of course, the same applies
  when $C_c(\Sigma;G)$ is viewed as compactly supported continuous
  covariant functions on $\Sigma$: in this case $\eps_x(f)=f(x)$.
\end{remark}

    For $x\in \unit{G}$, it will be useful to have a fixed orthonormal basis for
      $\sH_x$.   For $\dot\eta\in Gx$,  we select $\delta_{\dot\eta} \in \mathcal{H}_x$ such that

 \[(\varpi\circ\delta_{\dot
  \eta})(\dot\gamma)=
\begin{cases}1&\text{if $\dot\gamma=\dot\eta$}\\ 0&
  \text{otherwise.}\end{cases}\] 
  and insist in particular that $\delta_x(x)=[1,x]$.   Then \[ \{\delta_{\dot \eta}:
\dot\eta\in Gx\}\] is an orthonormal basis for $\mathcal{H}_x$.
In the sequel, we will have occasion to consider the element
$\delta_{\dot\eta}(\dot\eta)\in L$.  By choosing (and fixing) $\eta\in
q^{-1}(\dot\eta)$ there exists a unique $\lambda_{\dot\eta}\in \bbT$
such that 
\begin{equation}\label{onb1}
  \delta_{\dot\eta}(\dot\eta)=[\lambda_{\dot\eta}, \eta].
\end{equation}

It is sometimes useful to informally regard
$\innerprod{\pi_x(f)\delta_{\dot\eta},\delta_{\dot\zeta}}$ as 
a
product of elements of $L$, and we now give a formula which provides
this description.
 For $f\in
C_c(\Sigma;G)$, $x\in \unit{G}$ and $\dot\eta, \dot\zeta\in Gx$, the
definition of $\pi_x(f)$ and the inner product on $\sH_x$ yield
\begin{equation}\label{evalformula0}
\left[\innerprod{\pi_x(f)\delta_{\dot\eta},\delta_{\dot\zeta}},
                              x\right]
=\overline{\delta_{\dot\zeta}(\dot\zeta)}f(\dot\zeta\dot\eta^{-1})
\delta_{\dot\eta}(\dot\eta).
\end{equation} 
In particular,
 $\left[\innerprod{\pi_x(f)\delta_x,\delta_{\dot\zeta}},x\right]
 =\overline{\delta_{\dot\zeta}(\dot\zeta)} f(\dot\zeta).$  Therefore,
 \begin{equation}\label{evalformula}f(\dot\zeta)=\innerprod{\pi_x(f)\delta_x,\delta_{\dot\zeta}}\cdot\delta_{\dot\zeta}(\dot\zeta)\dstext{and}   \left|\innerprod{\pi_x(f)\delta_x,\delta_{\dot\zeta}}\right|=\varpi(f(\dot\zeta)).
 \end{equation}

\begin{example} \label{ctsco} Suppose that $\sigma$ is a normalized continuous
$2$-cocycle on the \'etale groupoid $G$.  This is a continuous
function from the set of composable pairs $G^{(2)}$ into $\bbT$ such
that $\sigma(\gamma, s(\gamma))=1=\sigma(r(\gamma),\gamma)$ and for
all composable triples, $(\gamma_1,\gamma_2,\gamma_3)$,
\[\sigma(\gamma_2,\gamma_3)\overline{\sigma(\gamma_1\gamma_2,\gamma_3)}\sigma(\gamma_1,\gamma_2\gamma_3)\overline{\sigma(\gamma_1,\gamma_2)}=1.\]
Define $\Sigma:=\T\times_\sigma G$, where
$\T\times_\sigma G$ is the Cartesian product of $\T$ and $G$ with the
product topology and multiplication defined by
$(z_1,\gamma_1)(z_2,\gamma_2)=(z_1z_2
\sigma(\gamma_1,\gamma_2),\gamma_1\gamma_2)$.  In this case, $L$ may be
identified with $\C\times G$ by $\phi:[\lambda,(z,\dot\gamma)]\mapsto
(\lambda z,\dot\gamma)$ and we identify sections of $L$
with functions on $G$ by 
\[\tilde f=p_1\circ\phi\circ f \quad \text{where } f\in C_c(G;\Sigma)\] where $p_1$ is the projection onto the first factor.  Now for
compactly supported sections $f,g$ of $L$,
\begin{align*} f*g(\dot\gamma)&=\sum
f(\dot\eta)g(\dot\eta\inv\dot\gamma)=\sum (\tilde f (\dot\eta),\dot\eta)(\tilde g(\dot\eta\inv\dot\gamma),\dot\eta\inv\dot\gamma)\\
&=\sum
[\tilde f (\dot\eta),(1,\dot\eta)][\tilde g(\dot\eta\inv\dot\gamma),(1,\dot\eta\inv\dot\gamma)]\\
&=\sum [\tilde f (\dot\eta)\tilde g(\dot\eta\inv\dot\gamma),
(\sigma(\dot\eta,\dot\eta\inv\dot\gamma),\dot\gamma)]\\ &=\sum
[\tilde f (\dot\eta)\tilde g(\dot\eta\inv\dot\gamma)\,
\sigma(\dot\eta,\dot\eta\inv\dot\gamma),(1,\dot\gamma)]\\
&=\left(\sum\tilde f(\dot\eta)\tilde g(\dot\eta\inv\dot\gamma)\,
\sigma(\dot\eta,\dot\eta\inv\dot\gamma),\dot\gamma)\right).
\end{align*}This last sum is the convolution formula for $\tilde
f,\tilde g$ in $\ccpair$ used by Renault in \cite{ren80}.  In
particular, if $\sigma$ is trivial then we get the usual convolution
formula for \'etale groupoid $C^*$-algebras.
\end{example}

We will use the following proposition to find useful subalgebras of the twisted groupoid $C^*$-algebra.

\begin{lemma}\label{smalltwist} Let $ \bbT\times\unit{G}\overset{\iota}\rightarrow
    \Sigma\overset{q}\rightarrow G$ be a twist and $H$ be an open subgroupoid of $G$.
    Define $\Sigma_H:=q\inv(H)$.  Then
    $$ \bbT\times\unit{H}\overset{\iota|_{\unit{H}}}\rightarrow \Sigma_H\overset{q|_{\Sigma_H}}\rightarrow H$$
    is a twist. 
 Moreover the map $\kappa: C_c(\Sigma_H;H)\hookrightarrow
\ccpair$ defined by extending functions by zero extends to an
inclusion of $C_r^*(\Sigma_H;H)$ into $\cpair$. \end{lemma}

\begin{proof} That
  $ \bbT\times\unit{H}\overset{\iota|_{\unit{H}}}\rightarrow
  \Sigma_H\overset{q|_{\Sigma_H}}\rightarrow H$ is a twist comes from
  the facts that $q(\iota(\lambda,x))=x$ and $\gamma\in \Sigma_H$ if
  and only if $\dot\gamma\in H$.

View elements of $C_c(\Sigma;G)$ and $C_c(\Sigma_H;H)$ as sections of
line bundles.
  By definition, the respective line
  bundles are $L_{\Sigma_H}=(\bbC\times \Sigma_H)/\T$ and
  $L_\Sigma=(\bbC\times \Sigma)/\T$.  Therefore, 
$L_{\Sigma_H}=L_\Sigma |_H$.  Since $H$ and
  ${\Sigma_H}$ are open, we may define
  $\kappa: C_c({\Sigma_H};H)\hookrightarrow \ccpair$ by extending
  functions by zero.

For each $x\in X$, let $\eps_x$ be defined as in Remark~\ref{altpix}
and let $\eps_x^H:=\eps_x\circ \kappa$.  Then $\eps_x$ and $\eps_x^H$
extend to states on $C_r^*(\Sigma;G)$ and $C_r^*(\Sigma_H;H)$.  Let
$(\pi_x,\sH_x)$ and $(\pi_x^H,\sH_x^H)$ be their associated GNS
representations and let $L_x\subseteq C^*_r(\Sigma;G)$ and
$L_x^H\subseteq C^*_r(\Sigma_H;H)$ be the left kernels of $\eps_x$ and
$\eps_x^H$ respectively.  For $h\in C_c(\Sigma_H;H)$,
$\eps_x^H(h^**h)=\eps_x(\kappa(h)^* *\kappa(h))$, so the map on
$C_c(\Sigma_H;H)$ defined by $(h+L_x^H)\mapsto (\kappa(h)+L_x)$ extends
to an isometry $W_x:\sH_x^H\rightarrow \sH_x$.  A calculation shows
that for $h\in C_c(\Sigma_H;H)$,
$W_x\pi_x^H(h)=\pi_x(\kappa(h)) W_x$ so that 
\[ W_x\pi_x^H(h)W_x^*=\pi_x(\kappa(h)) (W_xW_x^*).\]  Thus, for $h\in C_c(\Sigma_H;H)$, 
\begin{equation}\label{toobig}
  \norm{h}_{C_r^*(\Sigma_H;H)}=\sup_{x\in
    X}\norm{\pi_x^H(h)}=\sup_{x\in
    X}\norm{W_x\pi_x(\kappa(h))W^*_x}\leq
  \sup_x\norm{\pi_x(\kappa(h))}=\norm{\kappa(h)}_{C^*_r(\Sigma;G)}.
\end{equation}

Let $B=\overline{\kappa(C_c(\Sigma_H;H))}$, so $B$ is a
$C^*$-subalgebra of $C_r^*(\Sigma;G)$.  By~\eqref{toobig}, the map
$\kappa(h)\mapsto h$ extends to a
 $*$-epimorphism $\Theta:B\twoheadrightarrow
 C_r^*(\Sigma_H;H)$.

 Now let $\Delta: C_r^*(\Sigma;G)\rightarrow C_0(X)$ be the faithful
 conditional expectation
 determined by $C_c(\Sigma;G)\ni f\mapsto f|_X$; likewise let 
 $\Delta^H: C_r^*(\Sigma_H;H)\rightarrow C_0(X)$ be determined by
 $C_0(\Sigma_H;H)\ni h\mapsto h|_X$.  For $h\in C_0(\Sigma_H;H)$,
 $\Delta(\kappa(h))= \Delta^H(h)$. Therefore, for $b\in
 B$, $\Delta(b)=\Delta^H(\eta(b))$.     So for $b\in B$,
 $\Theta(b^*b)=0$ implies $b=0$ by the faithfulness of $\Delta$.   It
 follows that $\Theta$ is a $*$-isomorphism of $B$ onto
 $C^*_r(\Sigma_H;H)$.  Therefore,
 $\Theta^{-1}$ is a $*$-isomorphism of $C_r^*(\Sigma_H;H)$ onto
 $\overline{\kappa(C_0(\Sigma_H;H))}$, which is what we needed to show. 
\end{proof}

The following proposition allows us to view elements of
$C^*_r(\Sigma;G)$ as functions in $C_0(\Sigma;G)$.  This proposition
was originally proved in the case of Example~\ref{ctsco} above by
Renault in \cite[Proposition II.4.2]{ren80}.  Renault uses it without
proof in the full generality of twists in \cite{Ren08}.  As we know of
no proof of \cite[Proposition II.4.2]{ren80} for twists, we provide a proof here
at the level of generality we will require.
Note that $C_0(\Sigma,G)$ can be made into a Banach space with
$\norm{f}=\sup_{\dot\gamma\in G}\varpi(f(\dot\gamma))$.

\begin{prop}\label{r4.2} Let $(\Sigma; G)$ be a twist with $G$
\'etale.  Then the inclusion map $j: \ccpair \to C_0(\Sigma;G)$ extends to a
norm-decreasing injective linear map of  $C^*_r(\Sigma;G)$ into 
$C_0(\Sigma;G)$.
Moreover, the algebraic operations of adjoint and convolution on
$C_c(\Sigma;G)$ extend to  corresponding operations on
$j(C^*_r(\Sigma;G))$, that is, for every $a, b\in C^*_r(\Sigma;G)$ and
$\dot\gamma\in G$,
\begin{equation}\label{hprop}
  j(a^*)(\dot\gamma)=\overline{j(a)(\dot\gamma^{-1})}\dstext{and}
  j(ab)(\dot\gamma)=\sum_{r(\dot\eta)=r(\dot\gamma)}
j(a)(\dot\eta)\, j(b)(\dot\eta\inv\dot\gamma).
\end{equation}
\end{prop}

\begin{proof} The algebra $C_c(\Sigma;G)$ may be regarded as a
  subalgebra of $C^*_r(\Sigma;G)$ or as its image under $j$ in
  $C_0(\Sigma;G)$.    First we show that for $f\in
  \ccpair$ we have $\|f\|_r\geq
  \|f\|_\infty$.  To see this, for $\dot\gamma\in
  G$ consider $\delta_{s(\dot\gamma)}$.  We have
 
\begin{align}\label{jnorm}
\begin{split} \|f\|_r\geq\|\pi_{s(\dot\gamma)}(f)\|\geq
\|\pi_{s(\dot\gamma)}(f)\delta_{s(\dot\gamma)}\|&=\<\pi_{s(\dot\gamma)}(f)(\delta_{s(\dot\gamma)}),\pi_{s(\dot\gamma)}(f)(\delta_{s(\dot\gamma)})\>^{1/2}\\
&=\sqrt{\sum_{s(\dot\eta)=s(\dot\gamma)}|f(\eta)|^2}\geq
|f(\dot\gamma)|.
\end{split}
\end{align} Thus $j$ extends to a norm decreasing linear map
$j:C^*_r(\Sigma;G)\to C_0(\Sigma;G)$.

We turn to showing that $j$ is injective.  
Since $j$ is norm-decreasing, the equalities in~\eqref{evalformula}
extend to every element of $C_r^*(\Sigma;G)$.  Therefore, for any
$\dot\gamma\in Gx$, and $a\in C^*_r(\Sigma;G)$,
\[\norm{\pi_x(a)\delta_{\dot\gamma}}^2=\sum_{\dot u\in
    Gx}|\innerprod{\pi_x(a)\delta_{\dot\gamma}, \delta_{\dot
      u}}|^2=\sum_{\dot u\in Gx}|\pi_x(a)\delta_{\dot\gamma}(\dot
  u)|^2=|\pi_x(a)\delta_{\dot\gamma}(\dot\gamma)|^2=|j(a)(\dot\gamma\dot\gamma^{-1})|^2.\]
So if $j(a)=0$, then $\pi_x(a)=0$ for every
$x\in\unit{G}$.  Thus $a=0$, so $j$ is injective.

To verify the first equality in~\eqref{hprop}, observe that it holds
for $a\in C_c(\Sigma;G)$.  For general $a\in C^*_r(\Sigma;G)$, observe
that for any $f\in C_c(\Sigma;G)$, the fact that $j$ is contractive yields
\[\varpi(j(a^*(\dot \eta))-\overline{j(a)(\dot\eta)})\leq
  \varpi(j(a^*-f^*)(\dot\eta))+\varpi(\overline{j(f-a)(\dot\eta^{-1})})\leq 2\norm{a-f}_r.\]
As the right-most term in this inequality can be made as small as
desired by choosing $f$ appropriately, we obtain the first equality.

Before establishing the second, for $a\in C^*_r(\Sigma;G)$ and $x\in G^{(0)}$, define 
\[\norm{a}_{2,x}:=\norm{\pi_x(a)\delta_x}.\]  Then
$\max\{\norm{a}_{2,x}, \norm{a^*}_{2,x}\}\leq \norm{a}_r$ and 
\begin{align*}
  \norm{a}_{2,x}^2&=\sum_{\dot\eta\in
    Gx}|\innerprod{\pi_x(a)\delta_x,\delta_{\dot\eta}}|^2=\sum_{\dot\eta\in
                    Gx} |j(a)(\dot\eta)|^2\\
  \intertext{and, using the first equality in~\eqref{hprop},}
  \norm{a^*}_{2,x}^2&=\sum_{\dot\eta\in
                    xG} |j(a)(\dot\eta)|^2.
\end{align*}

To establish the second equality in~\eqref{hprop}, first note it holds when
$a,b\in \ccpair$.  Now let $a,b\in C^*_r(\Sigma;G)$ be arbitrary.
Suppose $(f_i), (g_i)$ are nets in $\ccpair$ such that
$\norm{f_i-a}_r\rightarrow 0$ and $\norm{g_i-b}_r\rightarrow 0$.
Then
\begin{align*}&\varpi\left(\sum_{r(\dot\eta)=r(\dot\gamma)} j(f_i)(\dot\eta) \,
    j(g_i)(\dot\eta^{-1}\dot\gamma) - \sum_{r(\dot\eta)=r(\dot\gamma)}
    j(a)(\dot\eta) \,
                j(b)(\dot\eta^{-1}\dot\gamma)\right)\\
              &= \varpi\left(\sum_{r(\dot\eta)=r(\dot\gamma)} j(f_i)(\dot\eta)
                \,j(g_i-b)(\dot\eta^{-1}\dot\gamma)
                +\sum_{r(\dot\eta)=r(\dot\gamma)}j(f_i-a)(\dot\eta)\,
                j(b)(\dot\eta^{-1}\dot\gamma)\right)\\
              &\leq
                \norm{f_i^*}_{2,r(\dot\gamma)}\norm{g_i-b}_{2,s(\dot\gamma)}
                +\norm{f_i^*-a^*}_{2,r(\dot\gamma)}\norm{b}_{2,s(\dot\gamma)}\\
              &\leq \norm{f_i}_r\norm{g_i-g}_r+\norm{f_i-a}_r\norm{b}_r,
\end{align*}
from which it follows that
\[    \lim_{i\rightarrow\infty}\sum_{r(\dot\eta)=r(\dot\gamma)} j(f_i)(\dot\eta) \,
    j(g_i)(\dot\eta^{-1}\dot\gamma) =\sum_{r(\dot\eta)=r(\dot\gamma)}
    j(a)(\dot\eta) \,
                j(b)(\dot\eta^{-1}\dot\gamma).\]
Therefore,
for every $\dot\gamma\in G$,
\begin{align*}j(ab)(\dot\gamma)
  &=\innerprod{\pi_{s(\dot\gamma)}(ab)\delta_{s(\dot\gamma)},
    \delta_{\dot\gamma}}\delta_{\dot \gamma}(\dot\gamma)=\lim  j(f_ig_i)(\dot\gamma)\\ &=\lim \sum_{r(\dot\eta)=r(\dot\gamma)}
f_i(\dot\eta)g_i(\dot\eta\inv\dot\gamma) =\sum_{r(\dot\eta)=r(\dot\gamma)}
a(\dot\eta)b(\dot\eta\inv\dot\gamma),
\end{align*}
as desired. 
\end{proof}

\begin{definition} Let $G$ be an \'etale groupoid and $\Gamma$ a
  discrete abelian group.   A \textit{twist graded by $\Gamma$} 
  is a twist
 $\bbT\times \unit{G}\hookrightarrow \Sigma\twoheadrightarrow
  G$ over $G$ together with continuous groupoid
  homomorphisms $c_\Sigma: \Sigma\rightarrow \Gamma$ and $c_G:
  G\rightarrow \Gamma$ such that the diagram, 
\begin{equation}\label{graddef0}\xymatrix{ \T\times \go\ar[r]& \Sigma
\ar[dr]_{c_{_\Sigma}} \ar[r] & G \ar[d]^{c_{_G}} \\ & &
\Gamma &}
\end{equation} commutes.
We will sometimes abbreviate~\eqref{graddef0} and simply say $\Sigma\rightarrow G$ is a \textit{$\Gamma$-graded
  twist}.  
\end{definition}

For $\omega\in \hat\Gamma$ and $t\in \Gamma$ we denote the natural
pairing $\omega(t)$ by $\langle\omega,t\rangle$.  We will use additive notation for the group $\Gamma$ and multiplicative notation for the group $\hat\Gamma$. We now show that the grading maps $c_\Sigma$ and $c_G$ induce an
action of $\hat\Gamma$ on $C^*_r(\Sigma;G)$.   This fact is well known to
experts but we include a proof for completeness.

\begin{lemma}\label{hatgammaact} Suppose $\Sigma\rightarrow G$ is a
  $\Gamma$-graded twist.  There exists a continuous action of
$\hat{\Gamma}$ on $C_r^*(\Sigma;G)$ characterized by
\[(\omega\cdot f)(\dot\gamma)=\<\omega,c_G(\dot\gamma)\>f(\dot\gamma)\]
where $\omega\in \hat{\Gamma}$ and $f\in \ccpair$.
\end{lemma}
\begin{proof} First we check that the action is multiplicative.  For
this we compute
\begin{align*}
(\omega\cdot f)*(\omega\cdot g)(\dot\gamma)&=\sum_{r(\dot\eta)=r(\dot\gamma)}(\omega\cdot f)(\dot\eta)(\omega\cdot g)(\dot\eta\inv\dot\gamma)\\
&=\sum_{r(\dot\eta)=r(\dot\gamma)}\<\omega,c(\dot\eta)\>f(\dot\eta)\<\omega,c(\dot\eta\inv\dot\gamma)\>g(\dot\eta\inv\dot\gamma)\\
&=\<\omega,c(\dot\gamma)\>\sum_{r(\dot\eta)=r(\dot\gamma)}f(\dot\eta)g(\dot\eta\inv\dot\gamma)=(\omega\cdot (f*g))(\dot\gamma).
\end{align*}

Now let $L$ be the line bundle over $G$ associated to $\Sigma$ and for
$x\in \go$ let $L_x:=L|_{Gx}$.  Consider the regular representation
$\pi_x$ of $C^*_r(\Sigma;G)$
associated to $x\in \go$.

For $\chi\in \mathcal{H}_x$ define $\chi_\omega\in \sH_x$ by $\chi_\omega(\dot\gamma):=
\overline{\<\omega,c(\dot\gamma)\>}\chi(\dot\gamma)$.   Then 
$\|\chi_\omega\|^2
=\|\chi\|^2$, so the mapping $\chi\mapsto \chi_\omega$ is a unitary
$W_\omega\in \mathcal B(\sH_x)$.

So for $f\in C_c(\Sigma;G)$,
\begin{align*}
\pi_x(\omega\cdot f)\chi(\dot\gamma)&=\sum_{r(\dot\eta)=r(\dot\gamma)}\<\omega,c(\dot\eta)\>f(\dot\eta)\chi(\dot\eta\inv
\dot\gamma)\\
&=\sum_{r(\dot\eta)=r(\dot\gamma)}\<\omega,c(\dot\gamma)\>\overline{\<\omega,c(\dot\gamma)\>
\<\omega,c(\dot\eta\inv)\>}f(\dot\eta)\chi(\dot\eta\inv \dot\gamma)\\
&=\<\omega,c(\dot\gamma)\>\sum_{r(\dot\eta)=r(\dot\gamma)}f(\dot\eta)\chi_\omega(\dot\eta\inv
\dot\gamma)=\<\omega,c(\dot\gamma)\>\pi_x(f)\chi_\omega(\dot\gamma).
\end{align*} This then implies that
$\|\pi_x(\omega\cdot f)\chi\|=\|\pi_x(f)\chi_\omega\|$.  So now
\[\|\pi_x(\omega\cdot f)\|=\sup_{\|\chi\|=1}\|\pi_x(\omega\cdot f)\chi\|=\sup_{\|\chi\|=1}\|\pi_x(f)\chi_\omega\|=\sup_{\|\chi\|=1}\|\pi_x(f)\chi\|=\|\pi_x(f)\|
\] and since this holds for all $x$ we get
\[\|\omega\cdot f\|_r=\|f\|_r\] as desired.

Now suppose that we have nets $\omega_i\to \omega$ and $a_i\to a\in C^*_r(G;\Sigma)$.  Consider
$\omega_i\cdot a_i-\omega\cdot a=\omega_i\cdot a_i-\omega_i\cdot a+\omega_i\cdot a-\omega\cdot a.$
Since $\|\omega\cdot a\|_r=\|a\|_r$, to show $\omega_i\cdot a_i\to \omega\cdot a$ it
suffices to show $\omega_i\cdot a\to\omega\cdot a$. For $i$ sufficiently large
we can assume $\norm{a}_r\sup|\<\omega_i-\omega\cdot c(\dot\eta)\>|<\epsilon$. Now

\begin{align*}
\|\pi_x(\omega_i\cdot a-\omega\cdot a)\chi\|&=\|\sum_{r(\dot\eta)=r(\dot\gamma)}\<\omega_i\omega\inv,c(\dot\eta)\>a(\dot\eta)\chi(\dot\eta\inv
\dot\gamma)\|\\
&=\|\<\omega_i\omega\inv,c(\dot\gamma)\>\sum_{r(\dot\eta)=r(\dot\gamma)}a(\dot\eta)\chi_{\omega_i\omega\inv}(\dot\eta\inv
\dot\gamma)\|\\ &\leq
|\<\omega_i\omega\inv,c(\dot\gamma)\>|\|a\|_r<\epsilon.
\end{align*} Since this holds for all $x\in \go$ we get the result.
\end{proof}

\begin{remark} \label{hatgammaact1}  When elements of $C_c(\Sigma;G)$
  are viewed as in~\eqref{ccsg2},  the action of $\hat\Gamma$ on
  $C_r^*(\Sigma;G)$ is characterized by
  \[(\omega\cdot f)(\gamma)=\innerprod{\omega, c_\Sigma(\gamma)}
    f(\gamma),\] where $\omega\in \hat\Gamma$ and $f\in C_c(\Sigma)$
  is covariant. 
\end{remark}


\section{$\Gamma$-Cartan pairs and abelian group actions} \label{gamma Cartan}
In this section we define the main objects of our study, \gc\ pairs, and explore the relationship between \gc\ pairs and strongly continuous actions of compact abelian groups on $C^*$-algebras. We first give some preliminary results on topologically graded $\cs$-algebras.

\begin{definition} A $\cs$-algebra $A$ is \emph{topologically graded} by a
  (discrete abelian) group $\Gamma$ if there exist a family of linearly
independent closed linear subspaces $\{A_t\}_{t\in\Gamma}$ of $A$ such that 
  \begin{itemize}
  \item $A_tA_s\subseteq A_{t+s}$,
  \item $A_t^*=A_{-t}$,
   \item  $A$  is densely spanned by  $\{A_t\}_{t\in\Gamma}$; and
\item there is
  a faithful conditional expectation from $A$ onto $A_0$.
\end{itemize}
\end{definition}

\begin{definition} Let $A$ be a $\cs$-algebra topologically graded by a group $\Gamma$. We call an element $a\in A$ \textit{homogeneous} if
$a\in A_t$ for some $t$.  Let $D \subseteq A_0$ be an
  abelian subalgebra.  We denote the set of normalizers of $D$ in $A$
  by $N(A,D)$ or simply $N$. Also, $n$ is a
  \textit{homogeneous normalizer} if it is both a normalizer and
  homogeneous: that is, $n$ is a normalizer and $n \in A_t$ for some
  $t \in \Gamma$.  We denote the set of homogeneous normalizers by
  $N_h(A,D)$ or simply $N_h$.
Notice that for $n\in N_h$ and $d\in D$ we have $nd, dn\in N_h$.  
\end{definition}

The term topologically graded
was introduced by Exel \cite{Exel97}; see also \cite{ExelBook}.  

An action of a compact abelian group on a \cstaralg\ produces a
topological grading, which we now describe in some detail.

Let $\Gamma$ be a
discrete abelian group and $A$  a \cstaralg.  As is customary, we
say \textit{$\hat \Gamma$ acts strongly on $A$} if there is a strongly
continuous group of automorphisms on the \cstaralg\ $A$ indexed by
$\hat{\Gamma}$. That is, there is a map
$\hat\Gamma\times A\rightarrow A$, written
$(\omega, a)\mapsto \omega\cdot a$ such that:
\begin{enumerate}
  \item for every $\omega$, $a\mapsto \omega\cdot a$ is an
    automorphism $\beta_\omega$ of $A$; \item the map $\omega\mapsto
    \beta_\omega$ is a homomorphism of $\hat \Gamma$ into $\aut(A)$; and 
    \item for each $a\in A$, the map $\omega\mapsto \omega\cdot a$ is
      norm continuous.
    \end{enumerate}  

Let $A^{\hat\Gamma}$ be the fixed point algebra
under this action.  For $t$ in $\Gamma$ and $a\in A$ define
\begin{equation}\label{phit} \Phi_t(a):=\int_{\hat{\Gamma}}
(\omega\cdot a) \<\omega\inv,t\> d\omega,
\end{equation}
and let 
\[ A_t=\Phi_t(A)
\] be the range of $\Phi_t$.
Then for each $t\in\Gamma$, $\Phi_t$ is a completely contractive and
idempotent linear map.  The following simple fact is worth noting.
\begin{lemma}\label{Phi0faith} The map  $\Phi_0: A\to A^{\hat{\Gamma}}=A_0$ is a faithful conditional
  expectation.
\end{lemma}
\begin{proof}[Sketch of Proof] That $\Phi_0$ is a conditional
  expectation is clear, so it remains to show $\Phi_0$ is faithful.
  If $\Phi_0(a^*a)=0$, then for every state $\rho$ on
$A$, $\int_{\hat\Gamma} \rho(\omega\cdot (a^*a))\, d\omega=0$.  Thus
$\rho(\omega\cdot (a^*a))=0$ for every state $\rho$ and every
$\omega\in\hat\Gamma$.  Taking $\omega$ to be the unit element gives
$\rho(a^*a)=0$ for every state, so $a^*a=0$.
\end{proof}

We now characterize the homogeneous elements
of $A$.  The following lemma is a generalization of  \cite[Lemma 5.2.10]{AASMBook}. where it is proved for $\Gamma=\mathbb{Z}$.

\begin{lemma}\label{At prop}Suppose $\hat\Gamma$ acts strongly on $A$. 
The
  following statements hold for all $t \in \Gamma, a,b \in A$.
\begin{enumerate}
\item\label{it1}$a\in A_t$ iff for every $\sigma\in
\hat\Gamma$, $\omega\cdot a=\<\omega, t\>a$.
\item \label{it2}$a\in A_t$ iff $a^*\in A_{-t}$.
\item\label{it3} If $a\in A_t$, $b\in A_s$ then $ab\in A_{t+s}$.
\item \label{it4} If $a\in A_t$ and $s\in \Gamma$, then
  $\Phi_s(a)=\begin{cases} a& \text{if $s=t$;}\\ 0&
    \text{otherwise.}\end{cases}$
\end{enumerate}
\end{lemma}

\begin{proof} Let $a\in A_t$ and $\sigma\in \hat\Gamma$.  Then
\begin{align*} \sigma\cdot a&=\sigma\cdot
\Phi_t(a)=\int_{\hat\Gamma}((\sigma\omega)\cdot a)
\innerprod{\omega\inv,t}d\omega=\int_{\hat\Gamma} (\omega\cdot a) \innerprod{\omega\inv,t
\sigma}d\omega\\ &=\<\sigma, t\>\Phi_t(a)=\<\sigma,t\>a.
\end{align*} Conversely if $\sigma\cdot a= \<\sigma, t\>a$ for every
$\sigma\in \hat\Gamma$, then
 \[ \Phi_t(a)=\int_{\hat\Gamma} (\omega\cdot a) \<\omega\inv,t\>
d\omega=\int_{\hat\Gamma} a \<\omega,t\>\<\omega\inv,t\> d\omega=a.
 \]
 
Items \eqref{it2} and \eqref{it3} follow immediately since
$\sigma\cdot (a^*)=(\sigma\cdot
a)^*=\overline{\innerprod{\sigma,t}}a^*=\<\sigma,-t\>a^*$ and $\sigma\cdot
(ab)=(\sigma\cdot a)(\sigma\cdot
b)=\innerprod{\sigma,t}a\<\sigma,s\>b=\<\sigma, t+s\>ab$.

Lastly for \eqref{it4}, \[\Phi_s(a)=\int_{\hat\Gamma} (\omega\cdot a)
\<\omega\inv,s\>d\omega=a\int_{\hat\Gamma}
\<\omega,t\>\<\omega\inv,s\>d\omega=\delta_{s,t}a.\qedhere\]
\end{proof}

The following lemma and its corollary show   the homogeneous spaces
$\{A_t\}_{t\in\Gamma}$ together densely span in $A$.  We thank  Ruy Exel for
showing us the simple proof.
\begin{lemma}\label{AtDense} Suppose the compact abelian group $\hat\Gamma$ acts strongly on the
  \cstar-algebra $A$, and $a\in A$.  Then $a\in\cspn\{\Phi_t(a): t\in \Gamma\}$.
  
\end{lemma}
\begin{proof}

Let $B:=\overline{\spn}\{\Phi_t(a): t\in \Gamma\}$.
Suppose $\rho$ is a bounded linear functional on $A$ which 
annihilates $B$.  Define $g_a:\hat\Gamma\rightarrow \bbC$
by $g_a(\omega)=\rho(\omega\cdot a)$.
Compute the Fourier transform of $g_a$: for $t\in \Gamma$,
\begin{align*}\hat g_a(t)&=\int_{\hat\Gamma}
                           g_a(\omega)\overline{\innerprod{\omega,t}}\,
                           d\beta\\
  &=\rho\left(\int_{\hat\Gamma} (\omega\cdot a)
                       \overline{\innerprod{\omega,t}}\, d\omega\right) \\
                     &=\rho(\Phi_t(a))=0.
\end{align*}  Since the Fourier transform is one-to-one, 
$g_a=0$.  Taking $\omega=1$, we get $\rho(a)=0$.  As this does not
depend on the choice of $\rho$,  by the Hahn-Banach
theorem, $a\in B$
\end{proof}

As an immediate corollary we get that $\{A_t\}_{t\in\Gamma}$ have dense span in $A$.  
\begin{cor}Suppose the compact abelian group $\hat\Gamma$ acts on the
  \cstar-algebra $A$.   For $t\in \Gamma$, let $A_t:=\{a\in A: \beta\cdot
  a=\innerprod{\beta,t}a \text{ for every } \beta\in \hat\Gamma\}$.  Then
  $A=\overline{\spn}\{A_t: t\in \Gamma\}$. 
  \end{cor}

\begin{remark}\label{tgchar}
Lemmas~\ref{At prop} and~\ref{AtDense} show
that if $\hat\Gamma$ acts strongly on $A$,
then $A$ is topologically graded by $\Gamma$.  In particular, when
$\Sigma\rightarrow G$ is a $\Gamma$-graded twist,
Lemma~\ref{hatgammaact} shows that $C^*_r(\Sigma;G)$ is topologically
graded by $\Gamma$.   In
\cite[Theorem~3]{Rae18} the converse to Lemma~\ref{At prop} is proved: it is
shown that if $A$ is topologically graded by $\Gamma$, then there is a
strongly continuous action of $\hat\Gamma$ on $A$ such that $a\in A_t$
if and only if
$$ a = \int_{\hat\Gamma} (\omega \cdot a) \<\omega^{-1},t\>\ d\omega. $$
\end{remark}

We now observe that the proof of Lemma~\ref{AtDense} can be used to
show that if $\cspn \, {N(A,D)}=A$ then $\cspn \,{N_h(A,D)}=A$. Here
are the details.  
\begin{proposition} \label{NhDense} Suppose $\hat\Gamma$ acts on $A$ and that $D$ is
  a MASA in $A_0$.  If $n\in N$, then for every $t\in\Gamma$,
  $\Phi_t(n)\in N_h$ and $n\in
  \overline{\spn}\{\Phi_t(n): t\in \Gamma\}$.
  \end{proposition}
\begin{proof}
Fix $n\in N$.  By Lemma~\ref{AtDense} it suffices to show $\Phi_t(n)\in N_h$.   Let $d\in D$.
Then $\Phi_t(n)^*d\Phi_t(n)\in A_0$.  For $e\in D$, and
$\omega\in\hat\Gamma$, $\omega\cdot e=e$. So

\begin{align*}
\Phi_t(n)^*d\Phi_t(nn^*n )e&=\Phi_t(n)^*d\int_{\hat\Gamma} \omega\cdot n(\omega\cdot(n^*ne))\<\omega\inv,t\>d\omega=\Phi_t(n)^*d\int_{\hat\Gamma} \omega\cdot(nen^*n)\<\omega\inv,t\>d\omega\\
&=\Phi_t(n)^*dnen^*\Phi_t(n)=\Phi_t(n)^*nen^*d\Phi_t(n)=\int_{\hat\Gamma}\omega\cdot(n^*nen^*)\<\omega,t\>d\omega d\Phi_t(n)\\
&=n^*ne\Phi_t(n)^*d\Phi_t(n)=en^*n\Phi_t(n)^*d\Phi_t(n)=e\Phi_t(nn^*n)^*d\Phi_t(n).
\end{align*}
This relation holds if we replace $n^*n$ by a polynomial in $n^*n$ and by taking limits we see that it holds if we replace $n^*n$ by $(n^*n)^{1/k}$ for any $k\in \N$.  
Since $\lim_k n(n^*n)^{1/k}=n$,  we find that
$\Phi_t(n)^*d\Phi_t(n)$ commutes with every element of $D$.  Since $D$
is a MASA in $A_0$, 
$\Phi_t(n)^*d\Phi_t(n)\in D$. A similar argument shows that
$\Phi_t(n)d\Phi_t(n)^*\in D$.  So $\Phi_t(n)\in N_h$.
\end{proof}

We now define a main object of study.
\begin{definition} \label{defGammaCartan} Let $A$ be $\cs$-algebra topologically graded by a
  discrete abelian group $\Gamma$ and $D$ an abelian $C^*$-subalgebra of
  $A_0$.  We say the pair $(A,D)$ is \gc\ if
\begin{enumerate}
\item $D$ is Cartan in $A_0$,
\item  $N(A,D)$ spans a dense subset of $A$.
\end{enumerate}
\end{definition}
The following observations are simple but important. In
particular, for \gc\
pairs we may  focus on
homogeneous normalizers in place of more general normalizers.

\begin{lemma}\label{nondeg}  Suppose $(A,D)$ is a \gc\ pair.
  The following statements hold. \begin{enumerate}
  \item The span of the homogeneous normalizers, $N_h(A,D)$, is dense in $A$.
    \item  If
  $(e_i)$ is an approximate unit for $A_0$, then $(e_i)$ is an
  approximate unit for $A$.
  \item For any $n\in N(A,D)$, $n^*n$ and $nn^*$ belong to $D$.
\item Any approximate unit for $D$ is an approximate unit for $A$.
  \end{enumerate}
\end{lemma}
\begin{proof}
  As noted in Remark~\ref{tgchar}, a topological grading arises
from an action of a compact abelian group.  By
Proposition~\ref{NhDense}, $N_h(A,D)$
spans a dense subset of $A$.

Now suppose $(e_i)$ is an (not necessarily countable) approximate unit for $A_0$.
Let $n\in N_h$.  Then $nn^*$ and $n^*n$
  belong to $A_0$.   Since $(e_i)$ is an approximate unit for $A_0$,
  \begin{equation}\label{nondeg1}
    (e_i n-n)(e_i n-n)^*=e_i nn^* e_i-nn^*
    e_i -e_i nn^*+nn^*\rightarrow 0,
  \end{equation}
  whence $e_i n\rightarrow
  n$.  Similarly, $ne_i \rightarrow n$.   Hence for any $a\in
  \spn N_h$, $e_i a\rightarrow a$ and $ae_i \rightarrow
  a$.  Since $\spn N_h$ is dense in $A$,
  $(e_i)$ is an approximate unit for $A$.

 Since $(A_0,D)$ is a Cartan pair, $D$ contains an approximate unit
 $(e_i)$ for $A_0$.  By part (2), $(e_i)$ is also an
 approximate unit for $A$.   Then for any $n\in N(A,D)$,
 $D\ni n^*e_i n\rightarrow n^*n$, so $n^*n\in D$. Likewise, $nn^*\in D$. 

Finally, if $(e_i)$ is an approximate unit for $D$ and $n\in N$,
\eqref{nondeg1} together with the fact that $nn^*\in D$, gives
$e_i n\rightarrow n$; likewise $ne_i\rightarrow n$.  As
before, $\overline{\spn N}=A$ implies $(e_i)$ is an approximate
unit for $A$.
\end{proof}


\section{Twists from $\Gamma$-Cartan pairs}\label{RI}

Throughout this section, we consider a fixed \gc\ pair $(A,D)$.
The purpose of this section is 
to define a twist
$\hat D \times \T \to \Sigma \to G$ from the pair $(A,D)$ so that $A \cong C^*_r(\Sigma;G)$ and
$D \cong C_0(\go)$. This task is completed in Theorem~\ref{mainthm1}.
Our methods follow those found in Kumjian~\cite{Kum86}
and Renault~\cite{Ren08}, and also use techniques from Pitts~\cite{Pitt12}.

Renault and Kumjian construct a twist from the Weyl groupoid
associated to a Cartan pair by first considering its groupoid $G$ of
germs and then using the multiplicative structure of the normalizers to
construct the twist as an extension $\Sigma$ of $G$ by $\bbT\times
\unit{G}$.  Finally, they recognize $\Sigma$ as a family of linear
functionals on $A$.

To a certain extent, we follow the Kumjian-Renault approach.  We will
define $\Sigma$ and $G$ in two ways.  We first construct  sets
$\Sigma$ and $G$ using the Weyl groupoid (the topologies and groupoid
operations come later).   After doing so,
we identify $\Sigma$ as a family of linear functionals and $G$ as as a
family of (non-linear) functions on $A$.  The product on $\Sigma$ and
$G$ is obtained by translating the product on $A$ to $\Sigma$
utilizing the first approach, and the second approach makes defining
the topologies on $\Sigma$ and $G$ straightforward.  Viewing $\Sigma$
and $G$ as functions highlights the
parallel between the Gelfand theory for commutative $C^*$-algebras and
relationship of the twist and the pair $(A,D)$ more transparent.

To begin, we fix some notation.
  Write
\[X:=\hat D.\]
We 
  generally identify $D$ with $C_0(X)$; thus for
  $x\in X$ and $d\in D$, we write $d(x)$ instead of $\hat d(x)$.

Let $E$ denote the faithful conditional expectation $E \colon A_0 \rightarrow D$.
By \cite{Rae18} there is a corresponding strong action of $\hat{\Gamma}$ on $A$.
We denote by $\Phi_t$ the completely contractive map $\Phi_t \colon A \rightarrow A_t$ as defined in Equation~\eqref{phit}.
Set \[\Delta:=E\circ \Phi_0.\]  By Lemma~\ref{Phi0faith},  $\Delta$ is a faithful
conditional expectation of $A$ onto $D$.

For $n\in N$, Lemma~\ref{nondeg} gives $n^*n, nn^*\in D$; let 
\[\dom(n):=\{x\in \hat D: n^*n(x)>0\}\dstext{and}\ran(n):=\{x\in \hat{D}:
  nn^*(x)>0\}.\] By the definition of normalizer, $ndn^*\in D$ for all
$d\in D$.  So $N_h$ acts on $D$ by conjugation.  As $D$ is abelian,
this induces a partial action $\alpha$ on the spectrum.  The following
result of Kumjian gives a precise description of this action.
\begin{prop}\cite[Proposition 1.6]{Kum86} \label{prop def alpha} Let
$n\in N$.  Then there exists a unique partial homeomorphism $\alpha_n:
\dom(n)\to \ran(n)$ such that for each $d\in D$ and $x\in\dom (n)$,
\begin{equation*} (n^*dn)(x)=d(\alpha_n(x))\, (n^*n)(x).
\end{equation*} 
\end{prop}
When the action is clear from the context, we will sometimes write,
\[n.x:=\alpha_n(x).\]  
By \cite[Lemma~4.10]{Ren08} (or \cite[Corollary~1.7]{Kum86}), for $n,m\in N$ and $d\in D$ we have 
\[
\alpha_n\circ\alpha_m=\alpha_{mn},\quad \alpha_{n^*}=\alpha_{n}\inv, \text{ and}\quad \alpha_d=\id_{\supp'(d)}.
\]The collection $\{\alpha_n:n\in N\}$ is an
inverse semigroup, sometimes called the \textit{Weyl semigroup} of the
inclusion $(A,D)$.

Dual to the Weyl semigroup is a collection of partial automorphisms
$\{\theta_n: n\in N\}$ of $D$.
Given $n\in N$, $\overline{nn^*D}$ and $\overline{n^*nD}$ are ideals
of $D$ whose Gelfand spaces may be identified with $\ran(n)$
and  $\dom(n)$ respectively. 
By~\cite[Lemma~2.1]{Pitts17}, the map
$nn^*D \ni d\mapsto n^*dn\in n^*D n$ extends uniquely to a
$*$-isomorphism
$\theta_n: \overline{nn^*D}\rightarrow \overline{n^*nD}$ such that for
every $d\in \overline{nn^*D}$, \begin{align}
                                 dn&=n\theta_n(d)\label{thetaintertw}
                                 \\
                                 \intertext{and for every
$x\in \dom(n)$,} 
                                 \theta_n(d)(x)&=d(\alpha_n(x)).  \label{thetaeval}
                               \end{align}

\begin{lemma}\label{exists h} Suppose $n\in N_h(A,D)$ and $x\in X$ such
that $\Delta(n)(x)\neq 0$.  Then $x$ is in the interior of the set of fixed
points of $\alpha_n$ and there exists $h\in D$ such that $h(x)=1 $ and
$nh=hn\in D$.
\end{lemma}

\begin{proof} First note $n\in A_0$ because
  $0\neq \Delta(n)(x)= E(\Phi_0(n))(x)$, and thus $\Phi_0(n)\neq 0$.
  Furthermore, $x\in \dom(n)$ and \cite[Lemma~2.5]{Pitt12} gives
  $\alpha_n(x)=x$.  We claim that $x$ is actually in the interior of
  the set of fixed points of $\alpha_n$.  If not, then there exists a net
  $(x_i)$ in $\dom(n)$ such that $\alpha_n(x_i)\neq x_i$ and
  $x_i\to x$.  Then $\Delta(n)(x_i)\to \Delta(n)(x)\neq 0$.  However,
  by \cite[Lemma~2.5]{Pitt12} again, $\Delta(n)(x_i)=0$ for all $i$, a
  contradiction.

Now let $F$ be the interior of the set of fixed points of $\alpha_n$ and
$J:=\{d\in D: \supp d\subseteq F \}$.  For $S \subseteq D$ let
$$ S^\bot = \{a\in D \colon ax=0 \text{ for all }x\in S\}. $$
Note that $J^{\bot \bot}$ is the fixed point ideal $K_0$ for $n$ (see
\cite[Definition~2.13]{Pitt12}).  Then by \cite[Lemma~2.15]{Pitt12}
there exists $h\in D$ with $h(x)=1$ and $nh=hn\in D'$.  But $n\in A_0$
and $h\in D$, so $nh\in A_0\cap D'= D$ because $D$ is maximal abelian
in $A_0$.  Thus $nh=hn\in D$.
This completes the proof.
\end{proof}

The following is an interesting structural fact about the relationship
between $\Delta$ and the action of $N_h$ on $D$, which is
used 
when defining the inverse operation on $\Sigma$.
\begin{proposition}\label{ndel} For any $n\in N_h$ and $a\in A$,
  \[n^*\Delta(a)n=\Delta(n^*an).\]
\end{proposition}
\begin{proof}
To begin,  we claim that for $m\in N_h$, 
\begin{equation}\label{ndel1}
n^*\Delta(nm)n=n^*n\Delta(mn).
\end{equation} 
Since the terms on both sides of \eqref{ndel1} belong to $D$, it 
suffices to show that for every $x\in X$
\begin{equation}\label{ndel2}
(n^*\Delta(nm)n)(x)=(n^*n\Delta(mn))(x).
\end{equation} 
As $n=\lim_{k\rightarrow \infty} n(n^*n)^{1/k}$, both sides of
\eqref{ndel2} vanish if $(n^*n)(x)=0$. Thus to obtain~\eqref{ndel2} it suffices to
prove that for $x\in \dom(n)$  
\begin{equation}\label{ndel3}
\Delta(nm)(\alpha_n(x))=\Delta(mn)(x),
\end{equation} and this is what we shall do.

Suppose first that $\Delta(nm)(\alpha_n(x))\neq 0$.  Lemma~\ref{exists
  h} shows there exists $k\in D$ with $k(\alpha_n(x))=1$ and
$nmk=knm\in D.$ Then \begin{align*} \Delta(nm)(\alpha_n(x))&=
                                                             (k\Delta(nm))(\alpha_n(x))=\Delta(knm)(\alpha_n(x))\\
                                                           &=
                                                             (knm)(\alpha_n(x))=\frac{(n^*(knm)
                                                             n)(x)}{(n^*n)(x)}
                                                             =\frac{\Delta((n^*kn)m n)(x)}{(n^*n)(x)}\\
                                                           &=
                                                             \frac{(n^*kn)(x)}{(n^*n)(x)}\Delta(mn)(x)=k(\alpha_n(x))\Delta(mn)(x)=\Delta(mn)(x).     \end{align*}

Next, suppose $\Delta(mn)(x)\neq 0$ and put $y=\alpha_n(x)$.  We do a
similar calculation.  Another
application of Lemma~\ref{exists h} produces $h\in D$ with $h(x)=1$
and $mnh=hmn\in D$.   As $x=\alpha_{n^*}(y)$,
\begin{align*}
  \Delta(mn)(x) &=
                  (mnh)(\alpha_{n^*}(y))=\frac{(n(mnh)n^*)(y)}{(nn^*)(y)}=
                  \frac{\Delta(nm(nhn^*))(y)}{(nn^*)(y)}
  \\
                &=\Delta(nm)(y)\frac{(nhn^*)(y)}{(nn^*)(y)}=\Delta(nm)(\alpha_n(x))h(\alpha_{n^*}(y))=\Delta(nm)(\alpha_n(x))h(x)\\
                &=
                  \Delta(nm)(\alpha_n(x)).
\end{align*}

We have  shown that $\Delta(mn)(x)\neq 0$ if and only if
$\Delta(nm)(\alpha_n(x))\neq 0$, and,  when this occurs,
$\Delta(mn)(x)=\Delta(nm)(\alpha_n(x))$.   Thus~\eqref{ndel3} holds,
completing the proof of the claim.

By varying $m$ and using the facts that $n^*n\in D$
(Lemma~\ref{nondeg}(3)) and $\overline{\spn N_h}=A$,~\eqref{ndel1} implies
that for every $a\in A$,
$n^*\Delta(na)n=(n^*n)\Delta(an)=\Delta(n^*(na)n)$. Therefore, for
every $a\in \overline{nA}$,
\[n^*\Delta(a)n=\Delta(n^*an).\]   Given $k\in \bbN$, there exists a sequence of  polynomials $\{p_j\}$  each of which  vanish
at the origin such that 
$(nn^*)^{1/k}=\lim_j p_j(nn^*)$.  Thus, for  $a\in A$  and $k\in\bbN$,
$(nn^*)^{1/k}a(nn^*)^{1/k}\in \overline{n A}$.  Hence for each
$a\in A$,
\begin{align*}
  n^*\Delta(a) n&=\lim_k n^*(nn^*)^{1/k}\Delta(a)(nn^*)^{1/k}n
                  =\lim_k n^* \left(\Delta((nn^*)^{1/k}a(nn^*)^{1/k})\right)n\\ &=\lim_k
                  \Delta(n^*(nn^*)^{1/k}a(nn^*)^{1/k}n) =\Delta(n^*an).
\end{align*}
This
completes the proof.
\end{proof}

\subsection{Local equivalence relations from homogeneous normalizers} Let
\[\mathcal{G}:=\{(n,x)\in N_h\times X: n^*n(x)\neq 0\}.\]  We now
define two equivalence relations on $\mathcal G$ arising 
as germs of the subsemigroup of the Weyl
semigroup arising from homogeneous normalizers.  While we shall define
the groupoids $\Sigma$ and $G$ in the twist $\Sigma\rightarrow G$ as
functions on $A$, the equivalence relations below will enable us to
define the multiplicative structure on $\Sigma$ and $G$.

\begin{definition}\label{equivrelR}For $(n,x), (n',x')\in
\mathcal{G}$.  Consider
\begin{enumerate}
\item[$(1)$] $x=x'$,
\item[$(2_\Sigma)$] there exist $d, d'\in C_c(X)$ such that
$d(x)>0$, $d'(x)> 0$ and $nd=n'd'$,
\item[$(2_G)$] there exist $d, d'\in C_c(X)$ such that $d(x)\neq 0$, $d'(x)\neq
0$ and $nd=n'd'$.
\end{enumerate} By~\eqref{thetaintertw} and~\eqref{thetaeval},
the latter two conditions may equivalently be replaced with the
following conditions.
\begin{enumerate}
\item[$(2_\Sigma{}')$] There exist $d, d'\in C_c(X)$ such that
  $d(\alpha_n(x)))>0$, $d'(\alpha_n(x)))>  0$ and $dn=d'n'$.
\item[$(2_G{}')$]  There exist $b, b'\in C_c(X)$ such that
  $d(\alpha_n(x)))\neq 0$, $d'(\alpha_n(x)))\neq
    0$ and $dn=d'n'$.
  \end{enumerate}
Note that in conditions $(2_\Sigma)$ and $(2_G)$, we may assume that
$d, d'\in \overline{n^*n D}\cap \overline{n'^*n' D}$; likewise we may
assume $d, d'\in \overline{nn^* D}\cap \overline{n'n'^* D}$ in
conditions $(2_\Sigma{}')$ and $(2_G{}')$.
  
Define $\sim_\Sigma$ as the relation given by $(1)$ and $(2_\Sigma)$ and $\sim_G$ as the relation given by $(1)$
and $(2_G)$.  We omit the proof that these are equivalence relations.
We denote the equivalence classes by $[n,x]_\Sigma,[n,x]_G$
respectively.  We shall omit the subscript when the proof does not depend on
which relation is used. Following Renault \cite{Ren08}, define
\[ \Sigma_{A,D, \Gamma}:=\mathcal{G}/\sim_\Sigma\quad\text{and}\quad G_{A,D, \Gamma}:=
\mathcal{G}/\sim_G.\] We omit the $ A, D, \Gamma$ from the notation and write
$\Sigma$ and $G$ respectively when the inclusion and grading are clear from
context.
\end{definition}

Essentially,  $G$ is a modification of the groupoid of germs of the
$\alpha$ action and $\Sigma$ is a twist on this.  
The following is a useful observation.
\begin{lemma} \label{A0} For $i=1,2$ suppose $(n_i,x_i)\in \mathcal{G}$ and
  $[n_1,x_1] =[n_2,x_2]$.  Then $n_1^*n_2\in A_0$.
\end{lemma}
\begin{proof}
  We do this only for $\sim_G$, leaving the obvious modifications for $\sim_\Sigma$ to
  the reader.  By definition of $\sim_G$, $x_1=x_2=:x$ and there exist
  $d_1, d_2\in D$ with $d_i(x)\neq 0$ and $n_1d_1=n_2d_2$.  Then
  $n_id_i\in N_h$ and
  $ d_1^*n_1^*n_2d_2=d_1^*n_1^*n_1d_1$ is a non-zero element of $D$.
  Since $n_i\in N_h$, there exists $t\in\hat\Gamma$ such that
  $n_1^*n_2\in A_t$.  But $A_t$ is a $D$-bimodule, so
  $d_1^*n_1^*n_2 d_2\in A_0\cap A_t$, whence $t=0$.
\end{proof}

 It is useful to have an alternative description of the equivalence
 relations $\sim_\Sigma$ and $\sim_G$ before continuing.

\begin{definition}\label{equivrelD} For $(n,x), (n',x')\in
\mathcal{G}$,   consider the properties
\begin{enumerate}
\item[$(i)$] $x=x'$,
\item[$(ii_\Sigma)$] $\Delta(n^*n')(x)>0$, and
\item[$(ii_G)$] $\Delta(n^*n')(x)\neq 0$.
\end{enumerate} Define $\approx_\Sigma$ to be the  relation
on $\mathcal{G}$ given by
$(i)$ and $(ii_\Sigma)$, and define $\approx_G$ be the relation on
$\mathcal G$ given by $(i)$
and $(ii_G)$.
\end{definition}

\begin{proposition}\label{R=D} The relations $\approx_\Sigma$ and
$\sim_\Sigma$ are the same.  Likewise, the relations $\approx_G$ and
$\sim_G$ are the same.
\end{proposition}

\begin{proof} We prove $\approx_\Sigma$ and $\sim_\Sigma$ are the
same.  The proof for $\approx_G$ and $\sim_G$ is similar.

Suppose $(n,x)\approx_\Sigma (n',x')$.  Then $x=x'$
and $\Delta(n^*n')(x)>0$. Since $n, n'$ are homogeneous normalizers,
so is $n^*n'$.
Lemma~\ref{exists h}
implies there is an $h\in D$ such that $hn^*n'=n^*n'h\in D$ and
$h(x)>0$.  Now consider the equalities,
\[ n(n^*n'h)((n'h)^*n'h)=nn^*[((n'h)(n'h)^*)n'h]=
  n'h(nn^*\circ\alpha_{(n'h)^*})(n'h)^*(n'h).\]
Take $d=(n^*n'h)((n'h)^*n'h)$ and
$d'=h(nn^*\circ\alpha_{(n'h)^*})(n'h)^*(n'h)$, so that $nd=n'd'$.  Note
\begin{align*}d(x)&=n^*n'(x)h(x)((n'h)^*n'h)(x)>0\text{ and}\\
d'(x)&=h(x)(nn^*\circ\alpha_{(n'h)^*})(x)((n'h)^*(n'h))(x)>0.
\end{align*} Thus $(n,x)\sim_\Sigma (n',x)$.  The converse follows immediately from the definitions.
\end{proof}

\subsection{Viewing $\Sigma$ as linear functionals} Our next goal is to show that $\Sigma$ may be identified as a family
of linear functionals on $A$ and $G$ as a family of functions on $A$.
This highlights the role of the inclusion $(A,D)$ in producing
$\Sigma$ and $G$ and will allow us to easily define Hausdorff
topologies on $\Sigma$ and $G$.  In addition, for $a\in A$ we will
define $\hat a: \Sigma\rightarrow \bbC$ by $\hat a([n,x])=[n,x](a)$.
The main result of this section shows that the map $a\mapsto \hat a$
is an isomorphism $A\ni a\mapsto \hat a\in C^*_r(\Sigma;G)$ which in a
natural sense extends the the Gelfand transform.

We write $\dual{A}$ for the Banach space dual of $A$.  For
$f\in \dual{A}$, let $f^*\in \dual{A}$ be defined by
$A\ni a\mapsto \overline{f(a^*)}$ and let $|f|$ be the function on $A$
defined by $|f|(a)=|f(a)|$.  For a non-empty subset
$K\subseteq \dual{A}$, write $|K|:=\{|f|:f\in K\}$.  Equip $K$ with
the relative weak-$*$ topology and $|K|$ with the quotient topology
arising from the surjective map, $K\ni f\mapsto |f|$.  Then $K$ and $|K|$ are Hausdorff.

Put \[\fS:=\{x\circ \Delta: x\in X\},\] so $\fS$ consists of all
states of the form $A\ni a \mapsto \Delta(a) (x)$.  Then $\fS$ is a
family of state extensions of pure states on $D$ to all of $A$.
We make the following observations.
\begin{dremark}{Observations} \label{obsS}
  \begin{enumerate}
  \item With the relative weak-$*$ topology (i.e.\ the
$\sigma(\dual{A},A)$-topology) on $\fS$,  the restriction map,
$\fS\ni \psi\mapsto \psi|_D$ is a homeomorphism of $\fS$ onto $X$.
\item 
 Lemma~\ref{exists h} implies that
if $\psi\in\fS$, then for every $n\in N_h$,
\begin{equation}\label{compat}
  |\psi(n)|^2\in\{0,\psi(n^*n)\}.
\end{equation}
This condition is a variant of
the notion of \textit{compatible state} introduced in~\cite{Pitts17},
the difference being that~\eqref{compat} is required to hold only for
elements of $N_h$ rather than all of $N$ as in~\cite{Pitts17}.

The compatibility condition~\eqref{compat} implies
(using~\cite[Proposition~4.4(iii)]{Pitts17} and the Cauchy-Schwartz
inequality) that in the GNS representation $(\pi_\psi, \sH_\psi)$
associated to $\psi$, the set of vectors
$\mathcal V:=\{n+L_\psi:n\in N_h\}$ has the property that any two
vectors in $\mathcal V$ are either orthogonal or parallel; here
$L_\psi$ is the left kernel of $\psi$,
$L_\psi:=\{a\in A: \psi(a^*a)=0\}$.  Notice also that
$\spn \mathcal V$ is dense in $\sH_\psi$.

\end{enumerate}
\end{dremark}

For any $(n,x)\in \mathcal G$, define an element of $\dual{A}$ by
\begin{equation}\label{eidef}
  \psi_{(n,x)}(a):=\frac{\Delta(n^*a)(x)}{|n|(x)}.
\end{equation}
Simple calculations
show that $\norm{\psi_{(n,x)}}=1$ and for $d_1,d_2\in D$ and $a\in A$,
\[\psi_{(n,x)}(d_1ad_2)=d_1(\alpha_n(x))\psi_{(n,x)}(a) d_2(x).\]  In other words, in
  the language of~\cite[Section~2]{DonsigPitts}, $\psi_{(n,x)}$ is a norm-one
  eigenfunctional with source $s(\psi_{(n,x)})=x$ and range
  $r(\psi_{(n,x)})=\alpha_n(x)$.   Furthermore, observe that $(n,x)\in \sG
  \Leftrightarrow (n^*,\alpha_n(x))\in \sG$ and a calculation using
  Proposition~\ref{ndel} shows that for $(n,x)\in \sG$,
  \begin{equation} \label{finv}
    \psi_{(n,x)}^*=\psi_{(n^*,\alpha_n(x))}.
  \end{equation}
For later use, notice that for $d\in D$ with $d(x)>0$ and $z\in \T$, 
  \begin{equation}\label{fsubrel}
  \psi_{(nd,x)}= \psi_{(n,x)}\quad\text{and}\quad  \psi_{(zn,x)}=\overline{z} \psi_{(n,x)}.
  \end{equation}

  Let
  \begin{equation}\label{defE}
    \mathcal E:=\{\psi_{(n,x)}:(n,x)\in\mathcal G\}.
  \end{equation}
  Since a state $\psi$ on a $C^*$-algebra $B$ is
  uniquely determined by $|\psi|$, it also makes sense to define source and
  range maps on $|\mathcal E|$ by $s(|\psi_{(n,x)}|)=x$ and
  $r(|\psi_{(n,x)}|)=\alpha_n(x)$. 
  Then the source
  and range maps carry $\mathcal E$ and $|\mathcal E|$ onto $X$.

  Given $\psi\in \sE$, write $\psi=\psi_{(n,x)}\in\sE$,  and choose $m\in N_h$
  such that $\psi(m)\neq 0$.  Notice that for any $a\in A$, we have
  \begin{equation}\label{sre}
    \Delta(a)(x)=\frac{\psi(ma)}{\psi(m)}\dstext{and}
    \Delta(a)(m.x)=\frac{\psi(am)}{\psi(m)}.
  \end{equation}
  (Indeed, since $\psi(m)\neq 0$, Lemma~\ref{exists h} gives
  $n\sim_G m$, and a computation gives~\eqref{sre}.) 
  Setting \begin{equation}\label{srdef}
    \fs(\psi):=\frac{\psi(ma)}{\psi(m)} \dstext{and}
    \fr(\psi)=\frac{\psi(am)}{\psi(m)},
  \end{equation}
  then $\fs(\psi)=s(\psi)\circ\Delta$ and
  $\fr=r(\psi)\circ \Delta$.  Thus $\fs(\psi)$ and $\fr(\psi)$ are the (necessarily unique) elements of $\fS$ satisfying
  \[\fs(\psi)|_D=s(\psi)\dstext{and} \fr(\psi)|_D=r(\psi).\]  

Also, notice that $\fS\subseteq \sE$, for if $\psi =x\circ \Delta\in
\fS$, then $\psi=\psi_{(d,x)}$ for any $d\in D$ with $d(x)>0$.  Also, it
follows easily (using Lemma~\ref{exists h}) that
  \begin{equation}\label{sSsubsE}
    \fS=\{\psi_{(n,x)}: \Delta(n)(x)>0\}=\{\psi_{(d,x)}: d\in D \text{ and }
    d(x)>0\}.
  \end{equation}
  
  We list a few additional properties of $\mathcal E$ and $|\mathcal E|$.
  \begin{lemma}\label{eigenprop}
    The following statements hold.
    \begin{enumerate}
    \item The map $\psi_{(n,x)}\mapsto
      [n,x]_\Sigma$ is a well-defined bijection of $\mathcal E$ onto
      $\Sigma$.  
\item If $g\in \sE$ and $m\in N_h$ satisfies $g(m)>0$, then
  $g=\psi_{(m,s(g))}$.  
\item $\mathcal E\cup\{0\}$ is weak-$*$ compact; in
            particular $\mathcal E$ is locally compact.  Furthermore, $s, r$
            are continuous mappings of $\mathcal E$ onto $X$.
\item The  map $|\psi_{(n,x)}|\mapsto [n,x]_G$
          is a well-defined bijection of $|\mathcal E|$ onto $G$.
\item If $|\phi|\in |\sE|$ and $m\in N_h$ satisfies $|\phi|(m)\neq 0$, then
  $|\phi|=|\psi_{(m,s(|\phi|))}|$.    
          \item $|\mathcal E|$ is locally compact and $s, r$
            are continuous mappings of $|\mathcal E|$ onto $X$.
          \end{enumerate}
        \end{lemma}
        \begin{proof}
We prove statements (1), (2) and (3), leaving the others to the reader.
To establish the first, it suffices to show $\psi_{(n_1,x_1)}=\psi_{(n_2,x_2)}$ if and only if $[n_1,
      x_1]_\Sigma=[n_2,x_2]_\Sigma$ and this is what we do.

          Suppose $\psi_{(n_1,x_1)}=\psi_{(n_2,x_2)}$.   Applying the source map gives
$x_1=x_2$; write $x:=x_1=x_2$.   Now
\[0<\psi_{(n_2,x)}(n_2)=\psi_{(n_1,x)}(n_2)=\frac{\Delta(n_1^*n_2)(x)}{|n_1|(x)},\]
which implies $\Delta(n_1^*n_2)(x)>0$.  Proposition~\ref{R=D} now gives
$[n_1,x_1]_\Sigma=[n_2,x_2]_\Sigma$.  Conversely, if
$[n_1,x_1]_\Sigma=[n_2,x_2]_\Sigma$, then $x_1=x_2=:x$ and there exists $d_1, d_2\in D$
with $d_1(x)>0$ and $d_2(x)>0$ such that $n_1d_1=n_2d_2$.  Then
\[\psi_{(n_1,x_1)}=\psi_{(n_1,x)}=\psi_{(n_1d_1,x)}=\psi_{(n_2d_2,x)}=\psi_{(n_2,x_2)}.\]
This gives statement (1).

For statement (2), write $\phi=\psi_{(n,x)}$ and apply Proposition~\ref{R=D} and part (1).

Turning now to statement (3),  $\fs(\psi_{(n,x)})=x\circ \Delta$
and $\fr(\psi_{(n,x)})=\alpha_n(x)\circ \Delta$, so the maps
$\fs, \fr:\sE\rightarrow \fS$ are surjective.  They are 
continuous by~\eqref{srdef}, so
$s, r: \sE\rightarrow X$ are also continuous surjections.

Next suppose that
$\psi_{(n_i, x_i)}$ is a net in $\mathcal E$ converging
weak-$*$ to $\phi\in\dual{A}$; write
$\psi_i:=\psi_{(n_i,x_i)}$.  If $\phi=0$, there is nothing to
do.  So suppose $\phi\neq 0$.  By~\cite[Proposition~2.3]{DonsigPitts},
$\phi$ is an eigenfunctional.
Thus if $x:=s(\phi)$, continuity of $s$ yields $x_i\rightarrow x$.

Since $\spn
N_h$ is dense in $A$, there exists $n\in N_h$ such that $\phi(n) > 0$.
Since $n(n^*n)^{1/k}\rightarrow n$ as $k\rightarrow \infty$, we have
$0<\phi(n)=\lim_k \phi(n(n^*n)^{1/k})=\phi(n)\lim_k (n^*n)^{1/k}(x)$.  Thus $n^*n(x)\neq 0$ and so
$(n,x)\in \mathcal G$.  Since $\psi_i\rightarrow \phi$,
$\psi_i(n)$ is eventually non-zero, so we may as well assume that
$\psi_i(n)\neq 0$ for every $\lambda$.  Proposition~\ref{R=D} implies
$(n_i,x_i)\sim_G(n,x_i)$.  Hence there exists
$z_i\in \bbT$ such that
$\psi_i=\psi_{(z_i n,x_i)}=\overline{z_i}\cdot
\psi_{(n,x_i)}$.   Therefore,
\[
0<\phi(n)=\lim
\overline{z_i}\frac{\Delta(n^*n)(x_i)}{|n|(x_i)}=\lim\overline{z_i}\,
  |n|(x_i).\]   As $|n|(x_i)\rightarrow |n|(x)$, we
conclude  $z_i\rightarrow 1$.  It follows that $\phi=\lim
\psi_i=\lim \overline{z_i}\cdot \psi_{(n,x_i)}=\psi_{(n,x)}$,
so $\phi\in \sE$.   Thus $\sE\cup\{0\}$ is a closed subset of the unit
ball of $\dual{A}$, and hence is compact.
 \end{proof}

\begin{dremark}{Notation}\label{gpoidftnals}
We use the bijections of Lemma~\ref{eigenprop} to identify $\Sigma$
with $\sE$ (respectively $G$ with $|\sE|$) and will use $\sE$ and $\Sigma$
interchangeably (resp.\ $G$ and $|\sE|$) depending upon what is
convenient for the context.    Thus for $a\in A$, we will often write
$[n,x]_\Sigma(a)$ and $[n,x]_G(a)$ instead of $\psi_{(n,x)}(a)$ and
$|\psi_{(n,x)}|(a)$.  Then $\Sigma$ and $G$ become Hausdorff
topological spaces of functions on $A$.   When convenient, we will also identify $\fS$
with $X$ via the restriction mapping from~\ref{obsS}(1).
\end{dremark}

\subsection{The twist associated to a $\Gamma$-Cartan pair} We are now prepared to place groupoid structures on $G$ and $\Sigma$.
This is done exactly as in~\cite[Definition~8.10 and
Theorem~8.12]{Pitt12};  for convenience, we provide sketches of the
proofs using the present notation.

\begin{lemma} $\Sigma$ and $G$ are Hausdorff topological groupoids
  under the following operations:
\begin{itemize}
\item Multiplication: $[m,\alpha_{n}(x)][n,x]=[mn,x]$;
\item Inversion: $[n,x]\inv=[n^*,\alpha_n(x)]$.
\end{itemize} The map $x\mapsto [d,x]$ for $d\in D$ with $d(x)>0$
identifies $X$ with the unit space of $\Sigma$ and $G$.  Furthermore,
under this identification $r([n,x])=\alpha_n(x)$ and $s([n,x])=x$.
\end{lemma}

\begin{proof}

 We sketch the proof for $\Sigma$.  The proof for $G$ is
 left to the reader (details may be found in~\cite{Pitt12}).
That inversion is well-defined and continuous follows from
\eqref{eidef},~\eqref{finv}, and Lemma~\ref{eigenprop}.  Also, it is
clear that inversion is involutive.

Next we show multiplication is well-defined.  Suppose $[m_1,y]_\Sigma=[m_2,y]_\Sigma$ and  $[n_1,x]_\Sigma=[n_2,x]_\Sigma$. Using the bijection in  Lemma~\ref{eigenprop} we can identify $\psi$ and $\phi$ with  $[m_1,y]_\Sigma=[m_2,y]_\Sigma$  and $[n_1,x]_\Sigma=[n_2,x]_\Sigma$ respectively.   We have $y=\alpha_{n_i}(x)$.  By the definition of $\sim_\Sigma$ we can assume $m_2=m_1 d$ and $n_2=n_1 d'$ where $d(y)>0$ and $d'(x)>0$.  So to show that multiplication is well defined it suffices to show that $\psi_{(m_1n_1,x)}=\psi_{(m_1dn_1d',x)}$.  But this follows since $m_1dn_1d'=m_1n_1\theta_{n_1}(d)d'$ and we know from equation~\eqref{fsubrel} that $\psi_{(\nu b, x)}=\psi_{(\nu,x)}$ for all $\nu\in N_h$, $x\in \dom(\nu)$ and $b\in D$ with $b(x)>0$.

Multiplication is associative since
multiplication in the $C^*$-algebra is.

Suppose $[m,x], [n,y]\in \Sigma$ are such that
  the composition 
  $[m,x]  [n,y]$ is defined.  Then $x=\alpha_n(y)$.  We must show that 
\[[m,x] [n,y] [n^*,\alpha_n(y)]=[m,x]\dstext{and}
  [m^*,\alpha_m(y)][m,x][n,y]=[n,y].\]
But these equalities follow from Lemma~\ref{eigenprop}(2) because
\[[mnn^*,\alpha_n(y)](m)>0\dstext{and} [m^*mn,y](n)>0.\]  
This completes the proof that $\Sigma$ is a
groupoid when equipped with the indicated operations.

Since $\unit{\Sigma}=\{[m,x]_\Sigma^{-1} [m,x]_\Sigma: (m,x)\in\mathcal G\}$ we obtain
\[\unit{\Sigma}=\{[d,x]_\Sigma\in \Sigma: d\in D \text{ and } d(x)>0\}=\fS.\]
It follows that  the map $X\ni x \mapsto [d,x]_\Sigma$ where $d\in D$ is chosen so that
$d(x)>0$, is a bijection of
$X$ onto $\unit{\Sigma}$.   Similarly, the map $X\ni x\mapsto [d,x]_G$
where $d\in D$ satisfies $d(x)>0$ (or merely satisfies $d(x)\neq 0)$) is a bijection of $X$ onto
$\unit{G}$.

For $(n,x)\in \mathcal G$,
$r([n,x])=[nn^*,\alpha_n(x)]$ and $s([n,x])=[n^*n,x])$.  This gives
the desired identification of the range and source maps.

We have already observed that inversion is continuous and we now verify that
multiplication is continuous.    Let $\sE^{(2)}$ be the set of
composable pairs, that is, the collection 
$(\psi,\phi)\in \sE\times \sE$ with $s(\psi)=r(\phi)$.
Suppose
$(\phi_i)_{i\in I}$ and
$(\psi_i)_{i\in I}$ are nets in $\sE$
converging to $\phi, \psi\in \sE$ respectively, and such that
$(\phi_i,\psi_i)\in \sE^{(2)}$ for all $\lambda$. Since
$s$ and $r$ are continuous, we find that
 $s(\phi)=\lim_i
 s(\phi_i)=\lim_i r(\psi_i)=r(\psi)$, so
 $(\phi,\psi)\in\sE^{(2)}.$  Let
$n,m\in N_h$ be such that $\phi(n)>0$ and $\psi(m)>0$.  There exists
 $i_0$, so that $i\geq i_0$ implies 
 $\phi_i(n)$ and $\psi_i(m)$ are non-zero.  For each
 $i\geq i_0$, there exist scalars
 $\lambda_i,\lambda'_i\in\bbT$ such that
 $\phi_i=\lambda_i[n,s(\phi_i)]$ and
 $\psi_i=\lambda'_i[m,s(\psi_i)]$.  Since
 $$\lim_i\phi_i(n)=\phi(n)=\lim_i
 [n,s(\phi_i)](n)
\dstext{and} 
 \lim_i\psi_i(n)=\psi(n)=\lim_i [n,s(\psi_i)](n),$$ we
 conclude that $\lim \lambda_i=1=\lim \lambda'_i$.  So for any $a\in
 A$,
\begin{align*}
(\phi\psi)(a)&=\frac{s(\psi)((nm)^*a)}{(s(\psi)((nm)^*(nm)))^{1/2}}=
\lim_i
\frac{s(\psi_i)((nm)^*a)}{(s(\psi_i)((nm)^*(nm)))^{1/2}}
=\lim_i [n,s(\phi_i)][m,s(\psi_i)]\\
&=\lim_i(\phi_i\psi_i)(a),
\end{align*} giving  continuity of multiplication.
\end{proof}

Define $q\colon\Sigma\to G$ and $\iota: \bbT\times \unit{G}\to \Sigma$ by
\[q([n,x]_{\Sigma}):=[n,x]_G\quad\text{and}\quad \iota(\lambda,
  [d,x]_G)=[\lambda |d|,x]_\Sigma.\]  Then $q$ and $\iota$ are
continuous groupoid homomorphisms with $q$ surjective and $\iota$ injective.
 Moreover,
\[ q^{-1}(\unit{G})=\{[d,x]_\Sigma \colon d\in D \text{ and }d(x)\neq
  0\}=\iota(\bbT\times\unit{G}).\]  Furthermore, for $(n,x)\in
\mathcal G$, and $\lambda\in \bbT$,
\[\iota(\lambda, [nn^*, \alpha_n(x)]_G)\, [n,x]_\Sigma= [\lambda n,x]_\Sigma=
  [n,x]_\Sigma\, \iota(\lambda,[n^*n,x]_G).\] We thus have a central
extension of groupoids,
\[\bbT\times\unit{G}\overset{\iota}\hookrightarrow
  \Sigma\overset{q}\rightarrow G.\]  Also, for $\lambda\in
\bbT$ and $(n,x)\in \mathcal G$, 
\begin{equation}\label{Tact1}
  \lambda \cdot [n,x]_\Sigma=[\lambda n,x]_\Sigma.
\end{equation}

As $\unit{G}$ may be identified with $X$, we usually identify
$\iota(\bbT\times \unit{G})$ with $\T\times X$ by
\begin{equation}\label{Dsigmaid} [d,x]_\Sigma\mapsto
\left(\frac{d(x)}{|d|(x)},x\right).
\end{equation} Under this identification,  the  extension  of
groupoids above becomes
\[\T\times X\hookrightarrow \Sigma\stackrel{q}{\to} G.\]

\begin{remark}\label{2cdot}
 We have already seen an action of $\bbT$ on $\Sigma$: 
$\lambda\cdot [n,x]_\Sigma=[\lambda n,x]_\Sigma$.  When elements of
$\Sigma$ are identified with their corresponding elements of $\sE$ via
the map in Lemma~\ref{eigenprop}, there is another action of $\bbT$ on
$\Sigma$, namely scalar multiplication of linear functionals.  These
actions differ: if scalar multiplication of linear functionals is
denoted by juxtaposition, then
\[\overline{\lambda} [n,x]_{\Sigma}=\lambda\cdot [n,x]_\Sigma.\]
\end{remark}

For $n\in N_h$, let
\[Z(n):=\{[n,x]_G: x\in \dom(n)\}.\]
\begin{lemma}\label{znb}
  For each $n\in N_h$, $Z(n)$ is an open bisection for $G$ and
  $\{Z(n): n\in N_h\}$ is a base for the topology on $G$.  Moreover,
  $q^{-1}(Z(n))$ is homeomorphic to $\bbT\times Z(n)$.  In
  particular, $G$ is an \'etale groupoid and the bundle
  $\Sigma\rightarrow G$ is locally trivial.
\end{lemma}
\begin{proof}
The relevant definitions and an
application of Lemma~\ref{R=D} yield
\[q^{-1}(Z(n))=\{[m,y]_\Sigma\in \Sigma: [m,y](n) \neq 0\},\] which is
an open subset of $\Sigma$.  Thus $Z(n)$ is an open subset of $G$.  We
claim $r|_{Z(n)}$ and $s|_{Z(n)}$ are homeomorphisms of $Z(n)$ onto
$\ran(n)$ and $\dom(n)$ respectively. As $r$ is the
  composition of the source map with the inversion map, it suffices to
  show this for $s$ only.  First note that
  $s|_{Z(n)}: Z(n)\to \dom(n)$ is a bijection by definition.  By
  Lemma~\ref{eigenprop}(6), $s,r: G\to X$ are continuous.  Next we
  show $(s|_{Z(n)})^{-1}$ is a continuous function from $\dom n$ to
  $Z(n)$.  If $x_i\in \dom n$ is a net and
  $x_i\rightarrow x\in \dom n$, then
  $\psi_{(n,x_i)}\rightarrow \psi_{(n,x)}$ (by definition of the
  weak-* topology), so $[n,x_i]_G\rightarrow [n,x]_G$ by
  Lemma~\ref{eigenprop}(4).  Thus the claim holds, and $Z(n)$ is
  therefore an open bisection.

Let $U\subseteq G$ be open and choose $[n,x]_G\in U$.  Then
$V:=U\cap Z(n)$ is an open bisection, so $s(V)$ is an open subset of
$X$ containing $x$.  Let $d\in D$ be such that $\supp d\subseteq s(V)$
and $d(x)=1$.  Since $\dom(nd)=\dom(n)\cap \supp'(d)\subseteq s(V)$,
we find
\[[n,x]_G\in Z(nd)=s^{-1}(\dom(nd))\subseteq V\subseteq U.\]   Thus, $\{Z(n): n\in
N_h\}$ is a base for the topology on $G$.  As $\{Z(n): n\in N_h\}$
covers $G$, $G$ is \'etale.

Consider the  map $\tau: \bbT\times \dom (n)\rightarrow q^{-1}(Z(n))$
defined by
$(z,x)\mapsto [zn,x]_\Sigma$.  This map is a  homeomorphism, and as
$s|_{Z(n)}:Z(n)\rightarrow \dom(n)$ is a homeomorphism, we see 
$\Sigma\rightarrow G$ is locally trivial.     
\end{proof}

The following summarizes our discussion so far.

\begin{proposition} \label{G top1}  $\Sigma$ and $G$ are
  locally compact Hausdorff topological groupoids, $G$ is \'etale, and
  $\bbT\times X\hookrightarrow \Sigma\overset{q}\twoheadrightarrow G$
  is a twist.
\end{proposition}

Define a map $\gr: N_h\to \Gamma$ by by taking $n\in A_t$ to $t$.  This
induces maps $c_\Sigma: \Sigma\to \Gamma$ and $c_G: G\to \Gamma$
given by
\begin{equation}\label{grdef}
  c_\Sigma([n,x]_\Sigma)=\gr(n)\dstext{and} c_G([n,x]_G)=\gr(n).
\end{equation}
Notice that the 
definition of the topologies and the groupoid multiplications imply
that  $c_\Sigma$ and $c_G$ are continuous homomorphisms.  We therefore
have produced the graded twist, 
\begin{equation*} \xymatrix{ \T\times \go\ar[r]& \Sigma
\ar[dr]_{c_{_\Sigma}} \ar[r] & G \ar[d]^{c_{_G}} \\ & &
\Gamma. &}
\end{equation*}

\subsection{Every $\Gamma$-Cartan pair is a twisted groupoid $C^*$-algebra} For $a\in A$, define a function $\hat a:\Sigma\rightarrow \bbC$ by
\[ [n,x]_\Sigma\mapsto \frac{\Delta(n^*a)(x)}{(n^*n)^{1/2}(x)},
  \dstext{that is,} \hat a([n,x]_\Sigma)=\psi_{(n,x)}\, (a).\] By
construction, $\hat a$ is a continuous function on $\Sigma$, and for
$z\in \bbT$, \[\hat{a}(z\cdot [n,x]_\Sigma)=\overline{z} \hat
  a([n,x]_\Sigma),\] so $\hat a$ may be regarded as a continuous
section of the line bundle over the twist $\Sigma\rightarrow G$.  Thus
we may regard the  open support of $\hat a$ as a subset of $G$, as in Remark~\ref{supp covariant}.

\begin{lemma}\label{supp n} Suppose $n\in N_h$, then the open support
of $\hat{n}$, $\supp'(\hat n)$, is the set $Z(n)$ 
\end{lemma}

\begin{proof} Consider $\hat{n}$ for $n\in N_h$.  Then

\[\hat{n}[m,y]=\frac{\Delta(m^*n)(y)}{(m^*m)^{1/2}(y)}.\] This is zero
unless $\Delta(m^*n)\neq 0$.  Now Proposition~\ref{R=D} gives
$[m,y]_G=[n,y]_G$, that is $[m,y]_G\in Z(n)$.
\end{proof}

\begin{lemma} \label{linear}The map $\Psi:A\to C(\Sigma;G)$ given by
  $a\mapsto \hat a$ is linear
and injective. 
\end{lemma}
\begin{proof} This map is linear since $\Delta$ is.  Injectivity will
follow since $\spn N_h$ is dense in $A$.  Indeed, suppose
$\hat{a}\equiv 0$.  Then for all $n\in N_h$,
\[\Delta(n^*a)(y)=0\] for all $y\in \text{Dom}(n)$.  Thus
$\Delta(n^*a)=0$ for all $n\in N_h$.  By assumption $a\in
\overline{\text{span}}(N_h)$; take a net $\nu_i\in \text{span}(N_h)$
such that $\nu_i\to a$.  By linearity,
\[ \forall i \, \Delta(\nu_i^*a)=0.  \] Thus by continuity
$\Delta(a^*a)=0$.  Since $\Delta$ is faithful, $a=0$.
\end{proof}

Now let 
\[N_{h,c}:=\{n\in N_h: \supp \hat n \text{ is compact}\} \dstext{and}
 A_c:=\spn N_{h,c} \text{ (no closure).}\]   Note that $A_c$ is a
$*$-algebra and by Lemma~\ref{supp n}, for
$a\in A_c$, $\hat a\in C_c(\Sigma;G).$

\begin{lemma}\label{dense subalgebra} Let $(A,D)$ be a \gc\ pair.  Then
\begin{enumerate}
\item\label{Nc dense} $N_{h,c}$ is dense in $N_h$.
\item\label{Ac dense} $A_c$ is dense in $A$.
\item\label{psi bijective} $\Psi: a\mapsto \hat{a}$ sends $A_c$
bijectively onto $C_c(\Sigma;G)$ and $D_c=D\cap A_c$ bijectively onto
$C_c(X)$.
\item\label{psi homo}$\Psi$ is a $*$-algebra homomorphism.
\end{enumerate}
\end{lemma}

\begin{proof} For \eqref{Nc dense}, let $(e_i)$ be an
  approximate unit for $D$ with $e_i\in C_c(X)$ for every
  $i$.  By Lemma~\ref{nondeg}(4), $(e_i)$ is also an
  approximate unit for $A$.   Thus $n=\lim ne_i$.  So to prove
$N_{h,c}$ is dense it suffices to show that $nd\in N_{h,c}$ for all $d\in
C_c(X)$.  Given $d\in C_c(X)$, let $d_1\in C_c(X)$ be such that
 $\supp'(d_1)\supseteq \supp(d)$.    
By
Lemma~\ref{supp n},
$\supp'(\widehat{nd_1})=Z(nd_1)$.  Recalling that $s|_{Z(nd_1)}$ is a
homeomorphism of $Z(nd_1)$ onto $\dom(nd_1)$, we see that
$\overline{Z(nd)}$ is compact because $\overline{Z(nd)}\subseteq
Z(nd_1)$ and $\dom(nd)=\supp'(n^*n)\cap  \supp'(d)$ has compact
closure in $\dom(nd_1)$.

Now \eqref{Ac dense} follows immediately from \eqref{Nc dense}.

Lemma~\ref{linear} shows that
$\Psi$ is injective and $\Psi(D_c)=C_c(\fS)\simeq C_c(X)$,  so to obtain \eqref{psi bijective}, we must show $\Psi(A_c)=C_c(\Sigma,G)$.
Now $ C_c(\Sigma;G)$ is the span of sections of the
line bundle supported on sets of the form $Z(n)$, as the $Z(n)$ form a
basis for $G$ and we can use a partition of unity argument.  Thus it
suffices to show that for  $n\in N_h$, every $f\in C_c(\Sigma;G)$ with support in
$Z(n)$ is in the image of $\Psi$.  To proceed note the following.
\begin{enumerate}[i.]
\item The line bundle is trivial over $Z(n)$: this is true because
  $q^{-1}(Z(n))=\{z\cdot [n,x]_\Sigma: x\in \dom(n), z\in \bbT\}$ and
  the map $\bbT\times Z(n) \ni (z,[n,x]_G)\mapsto [zn,x]_\Sigma$ is a
  homeomorphism of $\bbT\times Z(n)$ onto $q^{-1}(Z(n))$.
\item The source map of $G$ sends $Z(n)$ homeomorphically to $\{x:
n^*n(x)\neq 0\}$ because $Z(n)$ is an open bisection. 
\end{enumerate}

Now let $f$ be a section of the line bundle supported
on $Z(n)$.  By the first item above we can view $f$ as a function.  By
item (ii), $f=d\circ (s|_{Z(n)})$ for some $d\in D$. Now take
$a=\frac{n{\overline d}}{(n^*n)^{1/2}}$.  We show $\hat{a}=f$.  Indeed,
\[\hat{a}[m,x]=\frac{\Delta(dn^*m)(x)}{(n^*n)^{1/2}(x)(m^*m)^{1/2}(x)},\]
which is $0$ unless the germ of $m$ is the same as $n$.  So we can
assume that $n=m$. Hence the above becomes
\[\frac{d(x)n^*n(x)}{n^*n(x)}=d(x)=f([n,x])\] Thus $\hat{a}=f$ and the
claim holds.

Part~\eqref{psi bijective} now follows.

It remains to show \eqref{psi homo}.
By linearity it is enough to check that $\widehat{mn}=\hat{m}*\hat{n}$
and $\widehat{m^*}=(\hat{m})^*$ for $m,n\in N_{h,c}$.
Using~\eqref{covmul} we compute:
\[ \hat m * \hat n([\nu,x]_\Sigma)=\sum_{[v,y]_\Sigma
    [w,x]_\Sigma=[\nu,x]_\Sigma} \hat m([v,y]_\Sigma)\, \hat
  n([w,x]_\Sigma)\] (again, for each factorization
$[v,y]_G[w,x]_G=[\nu,x]_G$ only one factorization
$[v,y]_\Sigma[w,x]_\Sigma=[\nu,x]_\Sigma$ is chosen).  But
$ \hat m([v,y]_\Sigma)\, \hat n([w,x]_\Sigma)=0$ unless
$[v,y]_G=[m,y]_G$ and $[w,x]_G=[n,x]_G$.  When this occurs, there exist
$z_v, z_w\in \bbT$ so that $[v,y]_\Sigma=[z_vm, y]_\Sigma$ and
$[w,x]_\Sigma=[z_wn,x]_\Sigma$.  As
$[v,y]_\Sigma[w,x]_\Sigma=[\nu,x]_\Sigma$, we have $y=\alpha_n(x)$  and
$[\nu,x]_\Sigma=[z_vz_wmn,x]_\Sigma$.  So
\begin{align*}
  \hat m * \hat n([\nu,x]_\Sigma)&=\hat m([z_vm,y]_\Sigma)\hat
                                   n([z_wn,x]_\Sigma)\\
&=\overline{z_vz_w}\sqrt{(m^*m)(y) (n^*n)(x)} =\overline{z_vz_w}\left( \frac{(n^*m^*mn)(x)}{(n^*n)(x)}
                                                       (n^*n)(x)\right)^{1/2}\\
&=\widehat{mn}([z_vz_wmn,x]_\Sigma) =\widehat{mn}([\nu,x]_\Sigma),                                          \end{align*}
as desired.

To see $\widehat{m^*}=(\hat{m})^*$, use Lemma~\ref{eigenprop} and~\eqref{finv}:
\begin{align*}
  \widehat{m^*}([n,x]_\Sigma)&=\psi_{(n,x)}(m^*)=\overline{\psi_{(n,x)}^*(m)}
                        =\overline{\psi_{(n^*,\alpha_n(x))}(m)}\\
  &=\overline{\hat m( [n,x]_\Sigma^{-1})}=(\hat m)^*([n,x]_\Sigma).\qedhere
\end{align*}
\end{proof}

\begin{lemma}\label{isomPsi}
When $C_c(\Sigma;G)$ is equipped with the
reduced norm, $\Psi|_{A_c}: A_c\rightarrow
C_c(\Sigma;G)$ is an isometric $*$-isomorphism.  
\end{lemma}
\begin{proof}
Lemma~\ref{dense subalgebra} gives  $\Psi|_{A_c}$ is a $*$-isomorphism of $A_c$ onto $C_c(\Sigma;G)$.

Fix $x\in X$.  Using Remark~\ref{altpix}, we may regard $\pi_x$ as the
GNS representation of $C_c(\Sigma;G)$ arising from the functional
$\eps_x$.    On the
other hand, the state ${\rho_x}:=x\circ \Delta$ 
determines the GNS representation $(\pi_{\rho_x},\sH_{\rho_x})$ of $A$.   Let
$L_x\subseteq C^*_r(\Sigma;G)$ and $L_{\rho_x}\subseteq A$ be the left
kernels of $\eps_x$ and ${\rho_x}$ respectively. 

We claim that for $n,m\in N_{h,c}$, 
\begin{equation}\label{reps}
  {\rho_x}(mn)=\eps_x(\hat m *\hat n).
\end{equation}
To see this, choose $d\in D$
with $d(x)=1$, so that  $[d,x]_\Sigma\in \unit{\Sigma}$.  For
$[n,x]_\Sigma$ with $[n,x]_\Sigma^{-1}[n,x]_\Sigma=[d,x]_\Sigma$,
a computation similar to that used in the proof of
Lemma~\ref{dense subalgebra}\eqref{psi homo} gives
\begin{align*}
  \eps_x(\hat m *\hat n)&=\hat m *\hat n([d,x]_\Sigma)= \hat
                          m([n^*,\alpha_n(x)])\hat n([n,x]_\Sigma)\\
  &=\frac{\Delta(nm)(\alpha_n(x))}{\sqrt{(nn^*)(\alpha_n(x))}}
    (n^*n)(x)^{1/2}=\Delta(nm)(\alpha_n(x))\\
  &\overset{\eqref{ndel3}}=\Delta(mn)(x)={\rho_x}(mn).
\end{align*}

For $a\in A_c$,~\eqref{reps} gives
\[{\rho_x}(a^*a)=\eps_x(\widehat{a^*} * \hat a).\]  Thus for $a\in A_c$,
the map
$a+L_{\rho_x}\mapsto \hat a+L_x$ extends to an isometry $W_x:
\sH_{\rho_x}\rightarrow \sH_x$.  Lemma~\ref{dense subalgebra}\eqref{psi
  bijective} implies that $W_x$ is onto.   For $m, n\in N_h$,
\[W_x\pi_{\rho_x}(m)(n+L_{\rho_x})=W_x(mn+L_{\rho_x})=\hat m*\hat n +L_x=\pi_x(\hat
  m)W_x(n+L_{\rho_x}).\]  It follows that $W_x\pi_{\rho_x}(m)W_x^*=\pi_x(\hat m)$.
Hence for $a\in A_c$,
\[W_x\pi_{\rho_x}(a)W_x^*=\pi_x(\Psi(a)).\]

Finally, for $a\in A_c$,
\[\norm{\Psi(a)}_{C^*_r(\Sigma;G)}=\sup_x\norm{\pi_x(\Psi(a))}
  =\sup_x\norm{\pi_{\rho_x}(a)}=\norm{a}_A, \] with the last equality
following from the fact that $\Delta$ is faithful (as in the proof of
Lemma~\ref{linear}).
\end{proof}

We now come to the main result of this section.

\begin{thm}\label{mainthm1} Let $(A,D)$ be a \gc\ pair.
 Then there exists a graded twist
\begin{equation*} \xymatrix{ \T\times \go\ar[r]& \Sigma
\ar[dr]_{c_{_\Sigma}} \ar[r] & G \ar[d]^{c_{_G}} \\ & &
\Gamma &}
\end{equation*}
and a  $\hat\Gamma$-covariant $*$-isomorphism
$\Psi: A\to C^*_r(\Sigma;G)$ such that $\Psi(D)=C_0(\unit{G})$. \end{thm}

\begin{proof}
  Lemmas~\ref{dense subalgebra} and~\ref{isomPsi} show that $\Psi$
  determines a $*$-isomorphism of $A$ onto $C_r^*(\Sigma;G)$ and the
  construction of the graded twist shows $\Psi$ is $\hat\Gamma$-covariant.  It
  remains to show that $\Psi(D)=C_0(X)$.

  For
  $d\in D$ and $[n,x]_\Sigma\in \Sigma$,
\[\hat d([n,x]_\Sigma) =\frac{\Delta(n^*)(x)}{|n|(x)} d(x).\]
Changing perspective to viewing $\hat d$ as a section of the line
bundle instead of as a covariant function and recalling that
$[n,x]_G\in \unit{G}$ if and only if $\Delta(n)(x)\neq 0$, we get
\begin{align*}
  \hat d([n,x]_G)&=\left[\frac{\Delta(n^*)(x)}{|n|(x)} d(x),
                   [n,x]_\Sigma\right]\\
&\overset{\eqref{Tact},\, \eqref{Tact1}}=
    \begin{cases} 0 &\text{if $[n,x]_G\notin \unit{G}$}\\
      \left[d(x), [\frac{\Delta(n^*)(x)}{|n|(x)} n, x]_\Sigma\right] &
      \text{if $[n,x]_G\in \unit{G}$.}\end{cases}
\end{align*}
As $[\frac{\Delta(n^*)(x)}{|n|(x)} n,
  x]_\Sigma\in\unit{\Sigma}$, under the identification of $X$ with
  $\unit{\Sigma}$,  $\left[d(x),
  [\frac{\Delta(n^*)(x)}{|n|(x)} n, x]_\Sigma\right]$ and $[d(x),x]$
  represent the same element of the line bundle $L$.  Thus
                \[ \hat d([n,x]_G)=
    \begin{cases} 0 &\text{if $[n,x]_G\notin \unit{G}$}\\
      \left[d(x), x]\right] & \text{if $[n,x]_G\in \unit{G}$},\end{cases}
\] showing that $\Psi(d)\in C_0(X)$.  On the other hand, if $f\in
C_c(\Sigma;G)$ vanishes off $\unit{G}$, define 
$d\in D$  as follows.  For $x\in X$, choose
$n\in N_h$ so that $(n^*n)(x)\neq 0$; then let $d(x)$ be the unique
scalar  satisfying $f([n^*n,x]_G)=[d(x),x]\in L$. Then $\hat d(x)=f$.
It follows that $C_0(X)=\Psi(D)$ and the proof is complete.
\end{proof}

\section{$\Gamma$-Cartan Pairs from $\Gamma$-Graded Twists} \label{CoT}
In the previous section, we associated a
graded twist $(\Sigma,G,\Gamma)$ to a $\Gamma$-Cartan pair $(A,D)$ and
showed that $(A,D)$ can be recovered from $(\Sigma,G,\Gamma)$.  The purpose
of this section is to produce a $\Gamma$-Cartan pair from any 
suitable twist graded by the abelian group $\Gamma$.

Throughout this section we assume the following:
\begin{dremark}{Assumptions}\label{CoTAmpt}
   We fix a
$\Gamma$-graded twist
\begin{equation}\label{comm}\xymatrix{ \T\times \go\ar[r]& \Sigma
\ar[dr]_{c_{_\Sigma}} \ar[r]^q & G \ar[d]^{c_{_G}} \\ & &
\Gamma, &}
\end{equation}
with $G$ \'etale (and Hausdorff) where the diagram commutes and 
\begin{enumerate}
  \item $\Gamma$ is a discrete
    abelian group;
    \item $c_{_G}$ and $c_{_\Sigma}$ are (continuous) groupoid
homomorphisms; and  
  \item  $c_{_G}\inv(0)$ is effective. 
    \end{enumerate}
  \end{dremark}

The homomorphisms $c_{_G}$ and $c_{_\Sigma}$ are often called {\it
cocycles} in the literature as they are elements of the first groupoid
cohomology group with coefficients.  We will persist in referring to
$c_{_\Sigma}$ and $c_{_G}$ as cocycles here.

For notational convenience, let
  \[\Rho:=c\inv_\Sigma(0) \dstext{and}
R:=c\inv_G(0).\]  The commutativity of \eqref{comm} yields
\[\Rho=q\inv(R)\dstext{and} \unit{G}=\unit{R}.\]  Also, the continuity of $c_G$ ensures $R$ is a clopen
subgroupoid of $G$.
Thus we obtain the twist
  \begin{equation}\label{subt}
    \bbT\times \unit{G}\hookrightarrow
    \Rho\overset{q|_{\Rho}}\twoheadrightarrow R.
  \end{equation}
Since $R$ is \'etale and effective, it follows from
Renault's work in \cite[Section 4]{Ren08} that 
$C_0(\unit{G})$ is a Cartan MASA in $C^*_r(\Rho;R)$.  (Renault makes
the assumption that $\unit{R}$ is second countable, but that
assumption is not required to show that $(C^*_r(\Rho;R),
C_0(\unit{R}))$ is a Cartan pair. A close inspection of \cite{Ren08} shows that he uses $R$ effective instead of $R$ topologically principal, but these notions coincide when $\unit{R}$ is second countable.)

  By Lemma~\ref{smalltwist}, the inclusion 
$C_c(\Rho;R)\hookrightarrow C_c(\Sigma;G)$ given by extension by zero
extends to a $*$-monomorphism 
\[i: C^*_r(\Rho;R)\hookrightarrow \cpair.\]

\begin{prop} \label{fixedcartan}
The image of $C^*_r(\Rho;R)$ under $i$ is the fixed
point algebra of the action of $\hat\Gamma$ on $\cpair$, that is,
$i(C^*_r(\Rho;R))=\cpair^{\hat{\Gamma}}$. \end{prop}

\begin{proof} First notice that $i(C^*_r(\Rho;R))\subseteq
\cpair^{\hat{\Gamma}}$:  indeed,   if $f\in C_c(\Rho;R)$ and
$\dot\gamma\in G$,
\begin{align*}
  \omega\cdot i(f)(\dot\gamma)&=\<\omega,c_G(\dot\gamma)\>i(f)(\dot\gamma)=\begin{cases} 0 &
     \text{ if }\dot\gamma\notin R\\ \<\omega,0\>
     i(f)(\dot\gamma)=i(f)(\dot\gamma) &\text{ if }\dot\gamma\in
     R \end{cases}\\ &=i(f)(\dot\gamma).
\end{align*}

We now turn to showing
$\cpair^{\hat\Gamma}\subseteq i(C^*_r(\Rho;R))$.  First suppose
$a\in \cpair^{\hat{\Gamma}}\cap C_c(\Sigma;G)$.  Then for
$\dot\gamma\in G$,
\[a(\dot\gamma)=\int_{\hat\Gamma} (\omega\cdot a)(\dot\gamma)\, d\omega
  \overset{\eqref{hatgammaact}}=
  \int_{\hat\Gamma} \innerprod{\omega,c_G(\dot\gamma)}a(\dot\gamma)\,
  d\omega=
  a(\dot\gamma)\int_{\hat\Gamma} \innerprod{\omega,c_G(\dot\gamma)}\,
  d\omega, \] which vanishes unless $\dot\gamma\in R$.  Thus $a\in
i(C_c(\Rho,R))$.

For general $a\in C_r^*(\Gamma;G)^{\hat\Gamma}$, choose a net  $f_i\in C_c(\Sigma;G)$ so
that $f_i\rightarrow a$.  Then $\Phi_0(f_i)\rightarrow \Phi_0(a)=a$.
Also note that $\Phi_0(C_c(\Sigma;G))\subseteq C_c(\Sigma;G)$, so
$\Phi_0(f_i)\in i(C_c(\Rho;R))$.   It follows that $a\in
i(C^*_r(\Rho;R))$.
\end{proof}

Here is the main result of this section.
\begin{prop} \label{prop53} The pair $(C_r^*(\Sigma;G),C_0(\go))$ is a \gc\ pair.
\end{prop}

\begin{proof} It is well known that $C_0(\go)$ is an abelian
subalgebra of $C_r^*(\Sigma;G)$ that contains an approximate unit for
$C_r^*(\Sigma;G)$ \cite[Lemma~3.2]{Ren87}.  Lemma~\ref{hatgammaact}
gives an action of $\hat\Gamma$ on $C_r^*(\Sigma;G)$.  We have already
observed that $C_0(\unit G)$ is a Cartan MASA in $C^*_r(\Rho;R)$, so 
Lemma~\ref{fixedcartan} shows that $C_0(\unit{G})$ is a Cartan MASA in
$C^*_r(\Sigma;G)^{\hat\Gamma}$.   Since $G$ is \'etale, $\spn N$ is
dense in $C^*_r(\Sigma;G)$.  Thus $(C^*_r(\Sigma;G), C_0(\unit{G}))$ is \gc.
\end{proof}

We close this section with a result describing the supports of homogeneous normalizers.   This is necessary for the proof of Lemma~\ref{UpsDef} below. 

\begin{lemma}\label{opensupp a} Let $a\in \cpair$ and $S_a$ be the
open support of $a$.  Then $a$ is a homogeneous normalizer if and only
if $S_a$ is a bisection in $c\inv(t)$ for some $t$ in $\Gamma$.
\end{lemma}

\begin{proof} An element $a\in C^*_r(\Sigma;G)$
  is homogeneous of degree $t$ if and only if
  \begin{align*}
    a&=\int_{\hat\Gamma} \omega\cdot a\overline{\<\omega,t\>}d\omega\\
    &\Leftrightarrow
 a(\gamma)=\int_{\hat\Gamma}
\<\omega,c(\gamma)\>)a(\gamma)\overline{\<\omega,t\>}d\omega=a(\gamma)\int_{\hat\Gamma} 
      \<\omega,c(\gamma)\>\overline{\<\omega,t\>}d\omega\\
  &  \Leftrightarrow t=c(\gamma) \text{ for all $\gamma\in S_a$.}
  \end{align*}
  Thus $a\in A_t$ if
and only if $S_a\subseteq c\inv(t)$.

By the argument in \cite[Proposition~4.7]{Ren08}, if $a$ is a
normalizer, then $S_a\inv S^{}_a\subseteq \text{Iso}(G)$.  Thus, when $a$ is a
homogeneous normalizer, $S_a\subset c\inv(t)$ for some $t$, whence
$S_a\inv S_a\subseteq \text{Iso}(G)\cap R$.  As $R$ is effective  and $S_a\inv S_a$ is open, we obtain $S_a\inv S_a\subseteq \go$.
This implies that $S_a$ is a bisection.
\end{proof}

\section{Analysis of $\cpair$}\label{RD}

Throughout this section we fix a $\Gamma$-graded twist
$(\Sigma,G,\Gamma)$ satisfying Assumptions~\ref{CoTAmpt}.  Let
$(A,D)$ be the $\Gamma$-Cartan pair constructed from
$(\Sigma,G,\Gamma)$ in Proposition~\ref{prop53}.   An application of 
Theorem~\ref{mainthm1} to $(A,D)$ yields another $\Gamma$-graded twist
$(\Sigma_1, G_1, \Gamma)$ also satisfying Assumptions~\ref{CoTAmpt}.  Our goal is to show that
$(\Sigma_1,G_1,\Gamma)$ and $(\Sigma,G,\Gamma)$ are isomorphic in the
sense that there are topological groupoid isomorphisms
$\Upsilon_\Sigma: \Sigma\rightarrow \Sigma_1$ and
$\Upsilon_G: G\rightarrow G_1$ such that the diagram,
\begin{equation} \label{iso}\xymatrix{ 
    \bbT\times \unit{G}\ar[r]^{\iota}\ar[dd]_{\text{id}\times \Upsilon|_{\unit{G}}}& \Sigma \ar[dd]_{\Upsilon_\Sigma}\ar[rr]^q\ar[dr]_{c_\Sigma} && G\ar[dl]^{c_G}\ar[dd]^{\Upsilon_G}\\
    && \Gamma&\\
  \bbT\times \unit{G_1}\ar[r]_{\iota_1}&\Sigma_1\ar[rr]_{q_1}\ar[ur]^{c_{\Sigma_1}}&& G_1\ar[ul]_{c_{G_1}}\\
}
\end{equation}
commutes.

Throughout, we will use the notation established in Section~\ref{RI}
for 
$(A,D)$: thus $X=\unit{G}$, $D=C(X)$, $\Delta=E\circ \Phi_0$,
etc.  Further, notice that for $a\in
C_c(\Sigma;G)$, $\Delta(a)$ is nothing more than $a|_X$.  Thus, for
every $a\in A$, 
$\Delta(a)(x)=\eps_x(a)$.  
Lastly, recall from Section~\ref{RI} that $\sE=\{\psi_{(n,x)}:
(n,x)\in\sG\}$ is a family of linear functionals on $A$ which  becomes a
topological groupoid when equipped with the weak-$*$ topology,
 product $\psi_{(m,\alpha_n(x))}\psi_{(n,x)}=\psi_{(mn,x)}$,
and inverse $\psi_{(n,x)}^{-1}=\psi_{(n^*,\alpha_n(x))}$.    From Section~\ref{RI} we have $\Sigma_1=\sE$ and $G_1= |\sE|$. 

To begin, for $\gamma\in \Sigma$,  consider the linear functional
$\eps_\gamma$ on $A$ determined by Proposition~\ref{r4.2},
described as follows.  For $a\in
A$, there is a unique scalar $\eps_\gamma(a)$ such that
$j(a)(\dot\gamma)\in L$ is represented by $(\eps_\gamma(a),\gamma)\in
\bbT\times \Sigma$, that is,
\[j(a)(\dot\gamma)=[\eps_\gamma(a),\gamma].\]  Alternatively, if $a$ is
viewed as a covariant function on $\Sigma$,
\[\eps_\gamma(a)=a(\gamma).\]  Note that $\eps_\gamma$ is a norm-one
linear functional on $A$.

\begin{lemma}\label{UpsDef} The map $\Upsilon_\Sigma:
  \Sigma\rightarrow \dual{A}$ given by $\gamma\mapsto \eps_\gamma$
  is a homeomorphism of $\Sigma$  onto $\sE$.  Furthermore,
  $\Upsilon_\Sigma$ is an isomorphism of topological groupoids. 
\end{lemma}
\begin{proof}
  Fix $\gamma\in \Sigma$, put $x:=s(\gamma)$ and choose $n\in N_h$
  such that $\eps_\gamma(n)>0$.  By Lemma~\ref{opensupp a}, $n$ is supported
  on a homogeneous bisection, whence $\dot \gamma$ is the unique element of
  $\supp n$ whose source is $x$.  Thus 
  $(n^*n)(x)=\sum_{\sigma_1\sigma_2=x}
  \overline{n(\sigma_2)}n(\sigma_2)=\overline{n(\gamma)} n(\gamma)>0$,
  so $(n,x)\in \sG$.  To show  $\eps_\gamma=\psi_{(n,x)}$, it suffices to
  show $\eps_\gamma(m)=\psi_{(n,x)}(m)$ for every $m\in N_h$.   Choosing
  $m\in N_h$, we have $\Delta(n^*m)(x)=(n^*m)(x)=
\sum_{\dot
        \sigma\in Gx} \overline{n(\sigma)}m(\sigma)$.  As the terms in
      this sum are zero unless $\sigma\in \supp n$, and $\supp n\cap
      Gx=\{\dot\gamma\}$, we have  
  \[\psi_{(n,x)}(m)=\frac{\Delta(n^*m)(x)}{|n|(x)}
=\frac{\overline{n(\gamma)}m(\gamma)}{|n(\gamma)|}=\eps_\gamma(m)\]
because $n(\gamma)>0$.  Thus $\eps_\gamma=\psi_{(n,x)}$, as desired.  

Now suppose $(n,x)\in\sG$.   Since $n$ is supported on an open
bisection by Lemma~\ref{opensupp a}, there is a  unique element
of $\supp n$ whose source is  $x$.  Therefore, there is a unique
element  $\gamma\in\Sigma$ satisfying $s(\gamma)=x$ and
$\eps_\gamma(n)>0$.   The argument of the previous paragraph shows
$\psi_{(n,x)}=\eps_\gamma$.    We have thus shown that
$\Upsilon_\Sigma(\Sigma) = \sE$.  Notice that our work also shows that
$\Upsilon_\Sigma$ is bijective.  

Recall that $G$ is \'etale, $C_c(\Sigma;G)$ is dense in $A$,
and elements of $\sE$ are norm one linear functionals on $A$.  So if
$(\gamma_i)$ is a net in  $\Sigma$ and
$\gamma\in\Sigma$,
\begin{align*}
    \gamma_i\rightarrow \gamma
    &\Leftrightarrow \text{ for every $a\in C_c(\Sigma;G)$, $\eps_{\gamma_i}(a)\rightarrow
      \eps_\gamma(a)$}\\
      &\Leftrightarrow\text{  for every $a\in A$, $\eps_{\gamma_i}(a)\rightarrow
        \eps_\gamma(a)$ }\\
        &\Leftrightarrow \text{ $(\eps_{\gamma_i})$ converges weak-$*$ to
          $\eps_\gamma$.}
\end{align*}
Thus $\Upsilon_\Sigma$
is a homeomorphism.

We now observe that $\Upsilon_\Sigma$ preserves the groupoid
operations.  First, suppose $\gamma\in\Gamma$ and $n\in N_h$ is such
that $\eps_\gamma(n)>0$.  Then $\eps_\gamma=\psi_{(n,s(\gamma))}$.  For
$d\in D$, we have $\eps_\gamma(dn)=d(r(\gamma)) n(\gamma)$.  On the
other hand,
$\psi_{(n,s(\gamma))}(dn) =d(\alpha_n(s(\gamma))
\sqrt{(n^*n)(s(\gamma))}=d(\alpha_n(s(\gamma))) |n(\gamma)|$.  But as
this holds for every $d\in D$ and $n(\gamma)>0$, we conclude
that \begin{equation}\label{okunit}
  r(\gamma)=\alpha_n(s(\gamma)).
\end{equation}

Suppose the product of $\gamma_1, \gamma_2\in \Sigma$ is
defined.  For $i=1,2$, choose  $n_i\in N_h$ so that
$\eps_{\gamma_i}(n_i)>0$.  As $n_i$ are supported in open
bisections of $G$ (Lemma~\ref{opensupp a}), 
\[\eps_{\gamma_1\gamma_2}(n_1n_2)=(n_1n_2)(\gamma_1\gamma_2)=
  n_1(\gamma_1)n_2(\gamma_2)=\eps_{\gamma_1}(n_1)\eps_{\gamma_2}(n_2)>0.\]
We therefore obtain 
\[\Upsilon_\Sigma(\gamma_1\gamma_2)=
  \psi_{(n_1n_2,s(\gamma_2))}\overset{\eqref{okunit}}=\psi_{(n_1,r(\gamma_2))}\psi_{(n_2,s(\gamma_2))}=\Upsilon_\Sigma(\gamma_1)\Upsilon_\Sigma(\gamma_2).\]
For $\gamma\in\Sigma$ and $n\in N_h$ such that $\eps_\gamma(n)>0$, we
have
$\eps_{\gamma^{-1}}(n^*)=n^*(\gamma^{-1})=\overline{n(\gamma)}>0$, so
$\eps_{\gamma^{-1}}=\psi_{(n^*,r(\gamma))}$.  As
$r(\psi_{(n,s(\gamma))})=\alpha_n(s(\gamma))=$
$r(\gamma)=\alpha_n(s(\gamma))$, we obtain
\[\Upsilon_\Sigma(\gamma^{-1})=(\Upsilon_\Sigma(\gamma))^{-1}.\] 

Finally, suppose $z\in \bbT$ and $\gamma\in \Sigma$.  Choose $n\in N$
so that $\eps_\gamma(n)>0$.  Then $\eps_{z\cdot
  \gamma}(zn)=(zn)(z\cdot
\gamma)=\overline{z}(zn)(\gamma)=n(\gamma)>0$.  Thus,
$\eps_{z\cdot\gamma}=\psi_{(zn,s(\gamma))}$.  So by Remark~\ref{2cdot},
\[\Upsilon_\Sigma(z\cdot\gamma)=z\cdot\Upsilon_\Sigma(\gamma).\qedhere\]
\end{proof}

Writing $\Sigma_1:=\sE$ and $G_1:=|\sE|$, we thus have defined the two
left vertical arrows in~\eqref{iso}.  It follows that if $\dot\gamma\in G$, then for
$\sigma_1, \sigma_2\in q^{-1}(\dot\gamma)$,
$q_2(\Upsilon_\Sigma(\sigma_1))=q_2(\Upsilon_\Sigma(\sigma_2))$.
Therefore the map $\Upsilon_G:G\rightarrow G_1$ given by
\[\Upsilon_G(\dot\gamma):=q_1(\Upsilon_\Sigma(\gamma))\] is a 
well-defined isomorphism of groupoids.   That $\Upsilon_G$ is a
homeomorphism follows from the fact that $\Sigma\rightarrow G$ and
$\Sigma_1\rightarrow G_1$ are
locally trivial and $\Upsilon_\Sigma$
is a homeomorphism (or use the fact that $q$, $q_1$ are quotient maps
and $\Upsilon_\Sigma$ is a homeomorphism).

Now suppose $\dot\gamma\in G$ and
$c_G(\dot\gamma)=t$.  Then for $\gamma\in q^{-1}(\dot\gamma)$,
$c_\Sigma(\gamma)=t$.   Choosing $n\in N_h$ with $\eps_\gamma(n)>0$,
we obtain $\supp n\subseteq c^{-1}(t)$.   So by~\eqref{grdef} we
obtain
$c_{\Sigma_1}(\psi_{(n,s(\gamma))})=c_{G_1}(|\psi_{(n,s(\gamma))}|)=t$.
It follows that~\eqref{iso} commutes.  Thus we have proved the
following theorem.

\begin{thm}\label{mainthm2} Let $\Sigma\to G$ be a $\Gamma$-graded
  twist satisfying Assumptions~\ref{CoTAmpt}.  Let
$\Sigma_1\to G_1$ be the twist constructed from $(C_r^*(\Sigma;G),C_0(\go))$.
For each $\gamma\in \Sigma$, choose a homogeneous normalizer $n\in C^*_r(\Sigma,G)$ such
that $n(\gamma)>0$.  Then the map
\[ \Upsilon_\Sigma: \Sigma\to \Sigma_1\dstext{given by} \gamma\mapsto
[n ,s(\gamma)]_{\Sigma_1}
\] descends to a well-defined isomorphism of twists such that the
diagram~\eqref{iso} commutes.

\end{thm}

\begin{cor} Suppose $\Sigma\to G$ and $\Sigma'\to G'$ are
  $\Gamma$-graded twists satisfying Assumptions~\ref{CoTAmpt}.
  Suppose further that
\[\Xi: C_r^*(\Sigma;G)\to C_r^*(\Sigma';G')\] is an isomorphism of
$C^*$-algebras such that
\begin{enumerate}
\item $\Xi$ is equivariant for the induced $\hat\Gamma$ actions; and 
\item $\Xi|_{C_0(\go)}: C_0(\go)\to C_0({G'}^{(0)})$ is an
isomorphism.
\end{enumerate} Then there exists groupoid isomorphisms
$\upsilon_\Sigma, \upsilon_G$ such that the following diagram
commutes.

\begin{equation*}\xymatrix{
 \bbT\times\unit{G}\ar[r]^\iota \ar[dd]_{\text{id}\times v_G|_{\unit{G}}}   & \Sigma
\ar[rr]^q\ar[dr]^{c_{_\Sigma}}\ar[dd]_{\upsilon_\Sigma} & &
G\ar[dl]_{c_{_G}}\ar[dd]^{\upsilon_G} \\
& & \Gamma& \\
\bbT\times\unit{G'}\ar[r]_{\iota'}&\Sigma'\ar[ur]_{c'_\Sigma}\ar[rr]_{q'} & & G'\ar[ul]^{c'_G} }
\end{equation*}
\end{cor}

\begin{proof} By Theorem~\ref{mainthm2}, there are isomorphisms
$\Upsilon_\Sigma: \Sigma\to \Sigma_1$, $\Upsilon_G: G\to
G_1$, $\Upsilon_{\Sigma'}: \Sigma'\to \Sigma'_1$, $\Upsilon_{G'}: G'\to
G'_1$.  Since $\Xi$ is an equivariant isomorphism it takes
$\cpair^{\hat\Gamma}$ isomorphically onto $C^*_r(\Sigma';
G')^{\hat\Gamma}$.  Thus by construction, $\Sigma_1\cong \Sigma'_1$,
$G_1\cong G'_1$.  The result then follows from composition of
isomorphisms.
\end{proof}

\section{Examples}\label{ex}

\begin{example}Let $G$ be a finite discrete abelian group.  Take
$A=C^*(G)$.  Then 
\[C^*(G)=\spn\{\delta_g:g\in G\}.\] As pointed out in \cite{CRST}, if
$|G|=|H|$, then $C^*(G)\cong C^*(H)$.  So it is surprising that we would be
able to recover  $G$ using Theorem~\ref{mainthm2}.  However, as
this example illustrates, the induced action of $\hat\Gamma$ required
in Theorem~\ref{mainthm2} plays a crucial role.

Suppose $c: G\to \Gamma$ is a homomorphism of $G$ into a discrete
abelian group, with $c\inv(0)$ topologically principal.  Then
$c\inv(0)=\{0\}$, so $c$ is injective and $c(G)$ is isomorphic to
$G$ as a subgroup of $\Gamma$.  So $A_0=\C \delta_0$ and we consider
the inclusion $D=A_0\subseteq A$.

Notice that for $\omega\in \hat{G}$, $\omega\cdot \delta_g(h)=\innerprod{\omega,
c(h)}\delta_g(h)$ so that $\omega\cdot\delta_g=\<\omega,c(g)\>\delta_g$.
Thus, by Lemma~\ref{At prop}, $\delta_g\in A_{c(g)}$ and furthermore
$\delta_g \notin A_t$ for $t\neq c(g)$.  Since
$C^*(G)=\spn\{\delta_g:g\in G\}$ we have
\[ A_t=\begin{cases} \C \delta_g & t=c(g)\\ 0 &
\text{otherwise}.\end{cases}
\] Thus the homogeneous normalizers of $A_0$ are all of the form
$\lambda \delta_g$ for some $\lambda \in \C$.  Take $X=\{*\}=\hat
A_0$.  Now
\begin{align*} [\lambda
\delta_g,*]_\Sigma=[\lambda'\delta_{g'},*]_\Sigma &\Leftrightarrow
g=g' \text{ and } \overline{\lambda}\lambda'>0,\\ [\lambda
\delta_g,*]_G=[\lambda'\delta_{g'},*]_G &\Leftrightarrow g=g'.
\end{align*} So here $\Upsilon_G:g\mapsto [\delta_g,*]_G$ and
$\Upsilon_\Sigma (z,g)\mapsto [z\delta_g,*]_\Sigma$ where
$\Upsilon_\Sigma:\T\times G\to \Sigma_W$ are the desired isomorphisms
from Theorem~\ref{mainthm2}.
\end{example}

\begin{example}\label{THRG} Let $\Lambda$ be a $k$-graph. That is,
  $\Lambda$ is a small
category endowed with a functor, the \textit{degree map}, $d:\Lambda\to
\N^k$, that satisfies the following unique factorization property:  if $\lambda\in
\Lambda$ and $d(\lambda)=m+n$ there exists unique $\mu,\nu\in \Lambda$
such that $d(\mu)=m, d(\nu)=n$ and $\lambda=\mu\nu$.  We assume
$\Lambda$ has no sources in that for all objects $v$ and all $m\in
\N^k$ there exists $\mu$ with $r(\mu)=v$ and $d(\mu)=m$.

In \cite{KPS15}, Kumjian, Pask, and Sims introduce categorical
cohomology on a $k$-graph $\Lambda$.  In particular, they define a
$2$-cocycle with coefficients in $\T$ to be a function $\phi: \Lambda*
\Lambda\to \T$ such that
\[\phi(\lambda_1,\lambda_2)+\phi(\lambda_1\lambda_2,\lambda_3)=\phi(\lambda_2,\lambda_3)+\phi(\lambda_1,\lambda_2\lambda_3)\]
where $\Lambda*\Lambda:=\{(\mu,\nu):s(\mu)=r(\nu)\}$ and $\lambda_i$
are defined so that all of the compositions above make sense.  Denote
the set of these 2-cocycles by $Z_2(\Lambda,\T)$.  They prove in
\cite[Theorem~4.15]{KPS15}, that there is an isomorphism from the
second cubical cohomology group they defined in \cite{KPS12} to this
categorical cohomology group.

They define in \cite[Definition~5.2]{KPS15} the twisted $k$-graph
$C^*$-algebra by $\phi\in Z_2(\Lambda,\T)$ to be the universal
$C^*$-algebra $C^*(\Lambda,\phi)$ generated by elements $t_\mu$, $\mu \in \Lambda$ of a $C^*$-algebra satisfying the following.
\begin{enumerate}
\item The $t_v$ for $v\in d\inv(0)$ are mutually orthogonal projections,
\item $t_\mu t_\nu=\phi(\mu,\nu)t_{\mu\nu}$ whenever $s(\mu)=r(\nu)$,
\item $t^*_\lambda t^{}_\lambda=t_{s(\lambda)}$, and
\item\label{cuntz} for all $v \in d\inv(0)$ and $n \in \N^k$,  $\displaystyle t_v=\! \!\sum_{\substack{r(\lambda)=v  \\ d(\lambda) = n}} t^{}_\lambda t^*_\lambda$.
\end{enumerate} 

By the universal property of $C^*(\Lambda,\phi)$, $d:\Lambda\to \N^k$
induces an action of $\T^k$ on $C^*(\Lambda,\phi)$ characterized by
$z\cdot t_\mu^{} t_\nu^*=z^{d(\mu)-d(\nu)}t_\mu^{} t_\nu^*$.

Let
$$ C = \ol{\spn}\{t_\mu t_\nu^* \colon d(\mu) = d(\nu)\}, $$
and let
$$ D = \ol{\spn}\{t_\mu t_\mu^*\}. $$
By \cite[Lemma~7.4]{KPS15} $C$ is the fixed point algebra
$C^*(\Lambda,\phi)^{\T^k}$, for this action.  Moreover, as elements of
the generating set $\{t_\mu t_\nu^* \colon d(\mu) = d(\nu)\}$ are all
normalizers for $D$, to show $D$ is Cartan in $C$ it suffices to show
$D$ is maximal abelian and that there is a conditional expectation
from $C$ onto $D$.

The conditional expectation $P$ from $C$ onto $D$ is given by, for
$\mu,\nu \in \Lambda$ with $d(\mu) = d(\nu)$
$$ P(t_\mu t_\nu^*) = \delta_{\mu, \nu} t_\mu t_\nu^*. $$

The $\cs$-algebra $C$ is an AF-algebra.  This is shown in
\cite[Proposition 7.6]{KPS15}.  We recap and reframe some of those
details to show $D$ is Cartan in $C$.

For each $n\in \N^k$ let
$$ C_n = \ol{\spn} \{ t_\mu t_\nu^* \colon \mu, \nu \in \Lambda^n\}, $$
and let
$$ D_n = \ol{\spn} \{ t_\mu t_\mu^* \colon \mu \in \Lambda^n \}. $$
When $m\leq n$ we embed $C_m$ in $C_n$ using condition \eqref{cuntz}
in the definition of $C^*(\Lambda,\phi)$ above: if $\mu, \nu \in
\Lambda^m$ then
$$ t_\mu t_\nu^* = \sum_{\lambda \in \Lambda^{n-m}s(\mu)} t_\mu t_\lambda t_\lambda^* t_\nu^* \in C_n. $$
Note that this embedding also gives $D_m \subseteq D_n$.  We have then
that
$$ C = \ol{ \bigcup_{n\in \N^k} C_n}, $$
and
$$ D = \ol{ \bigcup_{n\in \N^k} D_n}. $$

For each $v \in \Lambda^0$ and $n\in \N^k$ denote by $K(\Lambda^nv)$
the compact operators on the Hilbert space $\ell^2(\Lambda^nv)$.
Using a matrix unit argument, it is observed in \cite[Proposition
7.6(1)]{KPS15} that
$$ C_n = \bigoplus_{v\in \Lambda^0} K(\Lambda^nv). $$ 
As $D_n$ is formed by the self-adjoint matrix units, $D_n$ is a
maximal abelian subalgebra of $C_n$.  Further there is a faithful
conditional expectation $P_n$ on $C_n$.  We can describe this
conditional expectation by
$$ P_n(a) = \sum_{\mu \in \Lambda^n} (t_\mu^{} t_\mu^*) a (t_\mu^{} t_\mu^*), $$
where the series convergences in the strong operator topology.  Using
this formula we can extend $P_n$ to all of $B(H)$.  A simple
calculation shows that $P_n(B(H)) = D_n'$, i.e., $P_n$ is a conditional
expectation onto $D_n'$.

We further note that the embeddings of $C_m$ into $C_n$ for $m\leq n$,
give $P_n|_{C_m} = P_m$.  The conditional expectation $P \colon C
\rightarrow D$ can then be described as the direct limit of the maps
$\{P_n\}$ (see e.g. \cite[Proposition A.8]{Rae}).

To show that $D$ is maximal abelian in $C$ we use an argument similar
to that found in \cite[Chapter~1]{StrVoic}.  Suppose $a\in D' \cap C$.
Since $a\in C$ there is a net $(a_n)$ with $a_n\in C_n$ such that
$$ \lim_n \| a_n - a\| = 0. $$
Further, since $a\in D'$, we have that $a \in D_n'$ for each $n\in
\N^k$, and thus $P_n(a) = a$ for each $n\in \N^k$.  Hence
$$ \|P_n(a_n) - a \| = \|P_n(a_n - a)\| \leq \|a_n - a\|. $$
And thus the net $(P_n(a_n))_n$ converges to $a$.  Since $P_n(a_n) \in
D_n$ it follows that $a\in D$, and therefore $D$ is maximal abelian in
$C$.

Thus $(C^*(\Lambda,\phi), D)$ is a $\mathbb{Z}^k$-Cartan pair.
Hence by Theorem~\ref{mainthm1} there exists a twist $\Sigma_W\to G_W$
such that $C^*(\Lambda,\phi)\cong C^*_r(\Sigma_W; G_W)$.  Notice that
here $\Sigma_W$ and $G_W$ consist of elements of the form $[s_\mu
s_\nu^*,x]$ with $\mu,\nu\in \Lambda$ and $x\in \Lambda^\infty$.

In \cite{KPS15}, the authors construct a groupoid $G_\Lambda$ and a
continuous cocycle $\varsigma$ such that $C^*(\Lambda,\phi)\cong
C^*_r(G_\Lambda,\varsigma)$ \cite[Theorem~6.7]{KPS15}.  By
Theorem~\ref{mainthm2}, $\Sigma_W\cong \T\times_\varsigma G_{\Lambda}$
and $G_W\cong G_\Lambda$, that is Theorem~\ref{mainthm1} recovers the
construction in \cite{KPS15}.  We provide some details of the
isomorphisms of twists given above, but to proceed we need to provide
a few details of the construction in \cite{KPS15}.

The groupoid construction in \cite{KPS15} is standard and goes back to
\cite{KP00, ren80, KPRR97}.  We say $\Lambda^\infty:=\{x:\N\times\N\to
\Lambda, x \text{ is a degree preserving functor}\}$ and
$\sigma^p:\Lambda^\infty\to \Lambda^\infty$ by
$\sigma^p(x)(m,n)=x(m+p,n+p)$.

\[G_\Lambda:=\{(x,\ell-m,y)\in\Lambda^\infty\times \Z^k\times
\Lambda^\infty: \ell,m\in \N^k, \sigma^\ell x=\sigma^m y\}\]

with the topology on $G_\Lambda$ given by basic open sets
\[ Z(\mu,\nu):=\{(\mu x, d(\mu)-d(\nu), \nu x): x\in \Lambda^\infty,
r(x)=s(\mu)=s(\nu)\}.
\]

It turns out that under our hypotheses each $Z(\mu,\nu)$ is compact
and open and so there exists a subset $\mathfrak{P}\subseteq
\Lambda\times \Lambda$ such that $\{Z(\mu,\nu):(\mu,\nu)\in
\mathfrak{P}\}$ is a partition of $G_\Lambda$.  Thus for each
$\gamma\in G_\Lambda$ there exists $(\mu_\gamma,\nu_\gamma)\in
\mathfrak{P}$ such that $\gamma\in Z(\mu_\gamma,\nu_\gamma)$.

Now by \cite[Lemma~6.3]{KPS15} for $\gamma,\eta\in G_\Lambda$
composable we can find $y\in \Lambda^\infty$, $\alpha,\beta,\zeta\in
\Lambda$ such that
\begin{align*}\gamma=(\mu_\gamma \alpha y,
&d(\mu_\gamma)-d(\nu_\gamma), \nu_\gamma\alpha y),
\quad\eta=(\mu_\eta\beta y, d(\mu_\eta)-d(\nu_\eta), \nu_\eta\beta y),
\\\gamma\eta&=(\mu_{\gamma\eta} \zeta
y,d(\mu_{\gamma\eta})-d(\nu_{\gamma\eta}), \nu_{\gamma\eta}\zeta y)
\end{align*} and
\[\varsigma_\phi(\gamma,\eta):=(\phi(\mu_\gamma,\alpha)-\phi(\nu_\gamma,\alpha)+(\phi(\mu_\eta,\beta)-\phi(\nu_\eta,\beta))-(\phi(\mu_{\gamma\eta},\zeta)-\phi(\nu_{\gamma\eta},\zeta))\]
is a well-defined continuous groupoid cocycle (see \cite{ren80} for
the definition).  We can then define
\[ \Sigma_{\Lambda,\phi}=\T\times_{\varsigma} G_\Lambda.
\] and then the twist is
\[ \T\times G_\Lambda^{(0)}\to \Sigma_{\Lambda,\phi}\to G_\Lambda.
\] Now the isomorphisms given in Theorem~\ref{mainthm2} are given by
\begin{align*} \Upsilon_\Sigma:& (z,(\mu_\gamma
x,d(\mu_\gamma)-d(\nu_\gamma),\nu_\gamma x))\mapsto
[s^{}_{\mu_\gamma}s_{\nu_\gamma}^*,\nu_\gamma x]_\Sigma\\ \Upsilon_G:&
(\mu_\gamma x,d(\mu_\gamma)-d(\nu_\gamma),\nu_\gamma x)\mapsto
[s^{}_{\mu_\gamma}s_{\nu_\gamma}^*,\nu_\gamma x]_G.
\end{align*}
\end{example}

\begin{example}\label{ex: BL}
A main result in \cite{BarlakLi} is
that any separable, unital, nuclear $\cs$-algebra which contains a
Cartan subalgebra satisfies the UCT. 
In fact, more is shown. 
It is shown in \cite[Theorem~3.1]{BarlakLi} that if $\Sigma \to G$ is a
twist, where $G$ is an \'etale Hausdorff locally compact second
countable groupoid where the reduced $\cs$-algebra $\cs_r(\Sigma;G)$ is
nuclear, then $\cs_r(\Sigma;G)$ satisfies the UCT. 
Hence we have the following corollary to Theorem~\ref{mainthm1} and
\cite[Theorem~3.1]{BarlakLi}.

\begin{cor} Let $A$ be a separable and nuclear $\cs$-algebra.
If $A$ contains an abelian subalgebra $D$ such that $(A,D)$ is a \gc\ pair for a some discrete abelian group $\Gamma$, then $A$ satisfies the UCT.
\end{cor}
\end{example}


\appendix
\section{Nonabelian groups} In the previous sections we assumed we had
gradings by discrete abelian groups and actions by the dual group, as
this case is familiar and doesn't involve the introduction of
coactions.  However all of the results of the paper can be extended to
gradings by nonabelian groups, by replacing actions with coactions: in
this short appendix we outline how to extend our results to the
nonabelian case for those readers already familiar with coactions.
For those readers interested in more information on coactions we
recommend \cite[Appendix~A]{EKQR05}.

Throughout this appendix, all tensor products of $C^*$-algebras are
spatial tensor products.

\stepcounter{equation}\textit{(\theequation) Uses of Commutativity.}
The alert reader will no doubt have noticed that we used the
commutativity of $\Gamma$
in a few key places:
\begin{enumerate}
\item to define a $\Gamma$ grading on a $C^*$-algebra $A$ when an
action of $\hat\Gamma$ on $A$ is given;
\item to define maps $\Phi_t: A\to A_t$, including the faithful
conditional expectation $\Phi_0$ onto $A_0$ the fixed point algebra
(Lemma~\ref{At prop});
\item to define an action of $\hat\Gamma$ on $\cpair$ where
$c:G;\Sigma\mapsto \Gamma$ (Lemma~\ref{hatgammaact}); and
\item to show the fixed point algebra of the action above contains
$C_0(\go)$ as a Cartan subalgebra (Proposition~\ref{fixedcartan}).
\end{enumerate}

Now suppose that $\Gamma$ is a not necessarily abelian discrete
group whose identity we denote by $e$.
For $s, t\in \Gamma$, we will use $\delta_s$ to denote the
indicator function of the set $\{s\}$ and $\delta_{s,t}$ for the
Kronecker $\delta$ (so $\delta_{s,t}=1$ if $s=t$ and $0$ if $s\neq
t$). Let $\Lambda:C^*_r(\Gamma)\to
B(\ell^2(\Gamma))$ be the left regular representation of $\Gamma$.
Also, the map $\delta_s\mapsto \delta_s\otimes \delta_s\in
C^*_r(\Gamma)\otimes C^*_r(\Gamma)$ extends to a $*$-homomorphism
$\coact_\Gamma: C^*_r(\Gamma)\rightarrow C^*_r(\Gamma)\otimes C^*_r(\Gamma)$.

Let $\coact: A\to M(A\otimes C^*_r(\Gamma))$ be a (reduced) coaction.  This means that $\coact$ is a
non-degenerate $*$-homomorphism of $A$ into $M(A\otimes
C^*_r(\Gamma))$ such that
\begin{enumerate}
  \item[(i)] $\coact(A)(I\otimes C^*_r(\Gamma))\subseteq A\otimes
    C^*_r(\Gamma)$; and 
    \item[(ii)]  $(\coact
      \otimes\text{id}_\Gamma)\circ\coact =(\text{id}_A\otimes
      \coact_\Gamma)\circ\coact$ (these maps belong to 
      $\mathcal{B}(A,  M(A\otimes
      C^*_r(\Gamma)\otimes C^*_r(\Gamma)))$.
    \end{enumerate}
   Notice that since $\Gamma$ is discrete, $I\otimes \delta_e$ is the
   identity of $M(A\otimes C^*_r(\Gamma))$.  Thus
   condition (i) implies that actually, 
\[\coact: A\rightarrow A\otimes C^*_r(\Gamma).\]
Furthermore, the fact that $\Gamma$ is discrete implies that $\coact$  is non-degenerate in the sense that $\coact(A)(I\otimes C^*_r(\Gamma))$ is dense in $A\otimes C^*_r(\Gamma)$, see~\cite{BS89} or \cite[Remark~A.22(3)]{EKQR05}. \par 
For $t\in \Gamma$, there is a slice map  $S_t:
A\otimes C^*_r(\Gamma)\to A$ characterized by
\[ a\otimes b\mapsto a\langle \Lambda(b)\delta_e, \delta_t\rangle,
\] (see \cite{Tom67} or \cite[\S A.4]{EKQR05}).  Further, it follows
from the second statement of~\cite[Lemma~A.30]{EKQR05} that for $x\in
A\otimes C^*_r(\Gamma)$, \[x=0 \Leftrightarrow \text{ for all $t\in
    \Gamma$, } S_t(x)=0.\]  Thus, every element $x\in A\otimes
C^*_r(\Gamma)$ has a uniquely determined ``Fourier series'',
\[x\sim \sum_{t\in \Gamma} S_t(x)\otimes \delta_t.\]  
Define continuous maps $\Phi_t: A \rightarrow A$ by
\[\Phi_t:=S_t\circ \coact.\] Then for every $a\in A$,
\begin{equation}\label{f.s.r.}
  \coact(a)\sim \sum_{t\in \Gamma} \Phi_t(a)\otimes \delta_t.
\end{equation}
While we will not need this fact here, the $*$-homomorphism
property of  $\coact$ and the  series
representation~\eqref{f.s.r.} implies that the ``coefficient maps''
$\{\Phi_t\}_{t\in\Gamma}$ behave much as Fourier
coefficients do under convolution 
multiplication and adjoints:  for every $t\in \Gamma$
and $a, b\in A$,
\[\Phi_t(ab)=\sum_{s\in
    \Gamma}\Phi_{ts^{-1}}(a)\Phi_s(b)\dstext{and}\Phi_t(a^*)=
  \Phi_{t^{-1}}(a)^*.\] 

What we do require is that condition (ii) in the definition of
coaction given above implies that for
every $s,t\in \Gamma$ and $a\in A$,
\[\Phi_s(\Phi_t(a))=\begin{cases} 0& \text{if $s\neq t$,}\\ \Phi_t(a)
    & \text{when $s=t$.}
  \end{cases} 
\]

Define
\[ A_t:=\Phi_t(A). \] Note that as $S_e$ arises from the (faithful) trace on
$C^*_r(\Gamma)$, $\Phi_e: A\to A_e$ is a faithful
conditional expectation.

We will call an element $a\in A$ \textit{homogeneous} if $a\in A_t$
for some $t\in \Gamma$.
This gives us the $\Gamma$-grading and an analog of Lemma~\ref{At
  prop}, which addresses the first two points of Paragraph~A.1.

We now address the third and fourth items of Paragraph~A.1.  Assume we
have a twist with a cocycle as in Section~\ref{CoT} but with $\Gamma$
not necessarily abelian:
\begin{equation}\label{commNC}\xymatrix{ \T\times \go\ar[r]& \Sigma
\ar[dr]_{c_{_\Sigma}} \ar[r] & G \ar[d]^{c_{_G}}\\ & &
\Gamma. }
\end{equation}
The proof of
\cite[Lemma~6.1]{CRST} goes through without change to show there
exists a coaction
\[\coact: \cpair \to \cpair \otimes C^*_r(\Gamma)\] characterized by
\[ \coact(f)=f\otimes \delta_t \quad\text{where $f\in C_c(G)$ and
$\supp(f)\subseteq c_G\inv(t)$}.
\] Note that for $f\in C_c(G)$ and $\supp(f)\subseteq c_G\inv(t)$,
\[ \Phi_s(f)=S_s(f\otimes \delta_t)=f \langle
\Lambda(\delta_t)\delta_e, \delta_s\rangle=f\langle
\delta_t,\delta_s\rangle =f \delta_{s,t}.
\]

Now for $f\in C_c(G)$ then $f=\sum_{t\in \Gamma} f|_{c_G\inv(t)}$ so
that the above computation show that
\[\Phi_s(f)=f|_{c_G\inv(s)}.\] By continuity of $\Phi_s$ and the map
$j: \cpair \to C_0(\Sigma; G)$ we get that
$j(\Phi_s(a))=j(a)|_{c_G\inv(s)}$.
As in Section~\ref{CoT}, let $R=c\inv_G(e)$ and $\Rho=c\inv_\Sigma(e)$.
Then Lemma~\ref{smalltwist} goes through without change and if $a\in
\cpair_e$ and a net $f_i\in \cpair$ with $f_i\to a$, by the above we have
$\Phi_e(f_i)\to a$ and $\Phi_e(f_i)=f_i|_R\in C_c(R;\Rho)$, so that
$a\in C^*(R;\Rho)$, giving Proposition~\ref{fixedcartan}.

\begin{thm}\label{mainthm1NA} Let $A$ be a $\cs$-algebra and let $D$
be an abelian $\cs$-algebra of $A$ such that: 
\begin{itemize}
\item there is a coaction $\coact$ of a discrete group $\Gamma$ on $A$;
\item $D$ is Cartan in the algebra $A_e$; and
\item $\ol{\spn} \ N_h(A,D) =A$.
\end{itemize}  
Let $A_c$ be the algebraic span of $N_h$.  Then
there exists a $\Gamma$-graded twist $\T\times X\hookrightarrow
\Sigma\to G$ and a $*$-isomorphism $\Psi: A_c\to C^*_r(\Sigma;G)$ which induces an isomorphism
$A\to C^*(\Sigma;G)$ taking $D$ to $C_0(\go)$.\end{thm}

\begin{proof}
The arguments of Section~\ref{RI} go through without change.
\end{proof}

\begin{remark} Compare the conditions of Theorem~\ref{mainthm1NA} to the definition of \gc\ (\ref{defGammaCartan}), noting that by Lemma~\ref{nondeg}~(1) the conditions on the normalizers coincide when $\Gamma$ is abelian.   While we have not checked details, we expect that if $\Gamma$ is a (not necessarily abelian) discrete group, and the hypotheses of Theorem~\ref{mainthm1NA} are weakened so that the condition $\overline{\spn N_h(A,D)}=A$ is replaced with $\overline{\spn N(A,D)}=A$, then it is still true that  $\overline{N_h(A,D)}=A$.  If this is the case, a $\Gamma$-Cartan pair could then be defined to be an inclusion of $C^*$-algebras $D\subseteq A$ satisfying the weakened hypotheses of Theorem~\ref{mainthm1NA}.  All the results of this paper would be valid for this notion of $\Gamma$-Cartan pairs.  \end{remark}

Likewise, the arguments of Sections~\ref{CoT} and~\ref{RD} yield the
following result for (possibly non-commutative) discrete groups $\Gamma$. 
\begin{thm}\label{mainthm2NA} Let $\Sigma\to G$ be a locally trivial
twist, with a cocycle $c$ into a discrete group $\Gamma$.  Then
$(C_r^*(\Sigma;G), C_0(\unit{G}))$ is a $\Gamma$-Cartan pair.  Let
$\Sigma_1\to G_1$ be the $\Gamma$-graded
twist constructed from $(C_r^*(\Sigma;G), C_0(\unit{G}))$.
For each $\gamma\in \Sigma$, choose a homogeneous normalizer $n\in C^*_r(\Sigma,G)$ such
that $n(\gamma)>0$.  Then the map
\[ \Upsilon_\Sigma: \Sigma\to \Sigma_1\dstext{given by} \gamma\mapsto
[n ,s(\gamma)]_{\Sigma_1}
\]
descends to a well-defined isomorphism of twists such that the
following diagram commutes.

\begin{equation*}\xymatrix{ & \Sigma
\ar[rr]^q\ar[dr]^{c_{_\Sigma}}\ar[dd]_{\Upsilon_\Sigma} & &
G\ar[dl]_{c_{_G}}\ar[dd]^{\Upsilon_G}\\ \T\times X
\ar[ur]^\iota \ar[dr]_{\iota_1}& & \Gamma&\\ &
\Sigma_1\ar[ur]_{c_{\Sigma_1}}\ar[rr]_{q_1} & &
G_1\ar[ul]^{c_{G_1}} }
\end{equation*}

\end{thm}


\begin{thebibliography}{00}

 \bibitem{AASMBook} G.~Abrams, P. Ara, and M. Siles Molina,
{\em Leavitt Path Algebras}, Lecture Notes in Mathematics, vol.~2192,
Springer, London, 2017. MR 3729290

\bibitem{ABHS16}P. Ara, J. Bosa, R. Hazrat, and A. Sims, {\em
Reconstruction of graded groupoids from graded Steinberg algebras},
preprint 2016. (arXiv:1601.02872v1 [math.RA]).

\bibitem{BS89}S. Baaj and G. Skandalis, {\em C*-alg\'ebres de Hopf et th\'eorie de Kasparov \'equivariante}, K-Theory {\bf 2} (1989), 683 --721

\bibitem{BarlakLi}S.  Barlak and X. Li, {\em Cartan subalgebras and the UCT problem.}
Adv. Math. {\bf 316} (2017), 748--769. 

\bibitem{BCFS}  J.Brown, L. Clark, C. Farthing, and A. Sims, \emph{Simplicity of
    algebras associated to \'etale groupoids}, Semigroup Forum \textbf{88} (2014),
    433--452.

\bibitem{BCH} J. Brown, L. Clark, and A. an Huef, {\em
Diagonal-preserving ring $*$-isomorphisms of Leavitt path algebras},
preprint 2015. (arXiv:1510.05309v3 [math.RA]).


\bibitem{BCW17}N. Brownlowe, T. Carlsen, and M. Whittaker, {\em Graph
algebras and orbit equivalence}, Ergodic Theory Dynam. Systems {\bf
37} (2017), 389--417.

\bibitem{Car} T.Carlsen, {\em $*$-isomorphism of Leavitt path algebras
over $\Z$}, preprint 2016.  (arXiv:1601.00777v2 [math.RA]).

\bibitem{CRS17} T. Carlsen, E. Ruiz, and A. Sims, {\em Equivalence and
stable isomorphism of groupoids, and diagonal-preserving stable
isomorphisms of graph $C^*$-algebras and Leavitt path algebras},
Proc. Amer. Math. Soc. {\bf 145} (2017), 1581--1592.

\bibitem{CRST} T. Carlsen, E. Ruiz, A. Sims, and M. Tomforde, {\em
Reconstruction of groupoids and $C^*$-rigidity of dynamical systems},
preprint 2017, (arXiv:1711.01052v1 [math.OA]).

\bibitem{CFST} L.O. Clark, C. Farthing, A. Sims and M. Tomforde, {\em
A groupoid generalisation of Leavitt path algebras}, Semigroup Forum,
\textbf{89} (2014), 501--517.

\bibitem{DonsigPitts} A. Donsig and D. Pitts, {\em Coordinates Systems and
    Bounded Isomorphisms,}  J. Operator Theory. {\bf 59}(2) (2008), 359--416.

\bibitem{EKQR05} S. Echterhoff, S. Kaliszewski, J. Quigg, and
I. Raeburn, {\em A Categorical Approach to Imprimitivity Theorems for
$C^*$-Dynamical Systems}, preprint 2005, (arXiv:0205322v2).




\bibitem{Exel97} R.~Exel, {\em Amenability for Fell Bundles}, J. Reine
Angew. Math., {\bf 492} (1997), 41--73. MR 1488064.
\bibitem{ExelBook} \bysame {\em Partial dynamical systems, Fell
bundles and applications}, Mathematical Surveys and Mono-graphs,
vol.~224, American Mathematical Society, Providence, RI, 2017.  MR
3699795

\bibitem{Exel08}
R.~Exel, \emph{Inverse semigroups and combinatorial {$C\sp \ast$}-algebras},
  Bull. Braz. Math. Soc. (N.S.) \textbf{39} (2008), no.~2, 191--313.
  \MR{2419901 (2009b:46115)}


\bibitem{FM77} J. Feldman and C. C. Moore, {\em Ergodic equivalence
relations, cohomology and von Neumann algebras I, II},
Trans. Amer. Math. Soc. {\bf 234} (1977), 289--359.

 \bibitem{Kum86} A. Kumjian, {\em On $C^*$-diagonals}, Can. J. Math.,
{\bf 38} (1986), 969--1008.
 
 \bibitem{KP00} A. Kumjian and D. Pask, \emph{Higher rank graph
{$C^*$}-algebras}, New York J. Math. \textbf{6} (2000), 1--20.
  
  \bibitem{KPRR97} A. Kumjian, D. Pask, I. Raeburn, and J. Renault,
{Graphs, Groupoids, and Cuntz-Krieger Algebras}, J. Funct. Anal. {\bf
144} (1997), 505--541.

\bibitem{KPS12}A. Kumjian, D. Pask, and A. Sims, {\em Homology for
higher-rank graphs and twisted $C^*$-algebras,} J. Funct. Anal. {\bf
263} (2012), 1539--1574.

\bibitem{KPS15} \bysame, {\em On twisted higher-rank graph
$C^*$-algebras,} Trans. Amer. Math. Soc.  {\bf 367} (2015),
5177--5216.

\bibitem{MM14} K. Matsumoto and H. Matui, {\em Continuous orbit
equivalence of topological Markov shifts and Cuntz-Krieger algebras},
Kyoto J. Math. {\bf 54} (2014) 863--877.

\bibitem{Mat13} K. Matsumoto, {\em Classification of Cuntz-Krieger
algebras by orbit equivalence of topological Markov shifts},
Proc. Amer. Math. Soc. {\bf 141} (2013), 2329--2342.

\bibitem{Mat12} H. Matui, {\em Homology and topological full groups of
\'etale groupoids on totally disconnected spaces},
Proc. Lond. Math. Soc. (3) {\bf 104 } (2012), 27--56.


\bibitem{Pitt12} D. Pitts, {\em Structure for regular inclusions},
preprint 2012, (arXiv:1202.6413v2 [math.OA]).

\bibitem{Pitts17}\bysame, {\em Structure for regular inclusions. I},
J. Operator Theory. {\bf 78}(2) (2017), 357--416.

\bibitem{Rae} I. Raeburn, {\em Graph algebras}. CBMS Regional Conference
Series in Mathematics, 103. \textit{Published for the Conference Board
of the Mathematical Sciences, Washington, DC; by the American
Mathematical Society, Providence, RI,} 2005.
\bibitem{Rae18} \bysame, {\em On graded $C^*$-algebras},
Bull.~Aust.~Math.~Soc.~97 (2018), no. 1, 127--132.  MR 3744874

\bibitem{ren80} J. Renault, \emph{A groupoid approach to
{$C^*$}-algebras}, Lecture Notes in Mathematics, vol. 793,
Springer-Verlag, New York, 1980.

\bibitem{Ren87}\bysame, {\em Repr\'esemtation des produits crois\'es
d'alg\`ebres de groupo\"ides}, J. Operator Theory {\bf 18} (1987),
67--97.

\bibitem{Ren08} \bysame, \emph{Cartan subalgebras in
{$C^*$}-algebras}, Irish Math.  Soc. Bulletin \textbf{61} (2008),
29--63.

\bibitem{SimsHaEtGrThC*Al}
A. Sims, \emph{Hausdorff \' etale groupoids and their {$C^*$-algebras}},
  arXiv:1710.10897v2, October 2018.

    \bibitem{Spi14} J. Spielberg. {\em Groupoids and $C^*$-algebras
for categories of paths,} Trans. Amer. Math. Soc. {\bf 366},
5771--5819.
    
    
    
    \bibitem{StrVoic} \c{S}. Str\u{a}til\u{a} and D. Voiculescu, 
{\em Representations of AF-algebras and of the group $U(\infty)$.} Lecture
Notes in Mathematics, Vol. 486. \textit{Springer-Verlag, Berlin-New
York,} 1975.
    
      \bibitem{Steinberg}B. Steinberg, {\em A groupoid approach to
inverse semigroup algebras}, Adv. Math. {\bf 223} (2010), 689--727.
  
\bibitem{TWW} A. Tikuisis, S. White, W. Winter, {\em Quasidiagonality of nuclear C$^*$-algebras}. Ann. of Math. (2) 185 (2017), no. 1, 229--284. 
    
    \bibitem{Tom07} M. Tomforde, {\em Uniqueness theorems and ideal
structure for Leavitt path algebras} J. Algebra {\bf 318} (2007),
270--299.

\bibitem{Tom67} J. Tomiyama, {\em Applications of Fubini type theorem to the tensor products of $C^*$-algebras}, T\^{o}hoku Math. Journ. {\bf 19} (1967), 213--226.

\bibitem{WillardGeTo}
S. Willard, \emph{General topology}, Addison-Wesley Publishing Co.,
  Reading, Mass.-London-Don Mills, Ont., 1970. \MR{0264581 (41 \#9173)}


\end{thebibliography}
\bibliographystyle{amsplain}


\end{document}